\documentclass[reqno, a4paper]{amsart}
\usepackage{amsmath, amssymb, enumerate, esint,bbm}
\usepackage{datetime}
\usepackage[colorlinks=true,linkcolor=blue, citecolor=red]{hyperref}
\usepackage[nameinlink,capitalize]{cleveref}
\usepackage{dsfont}

\crefname{equation}{}{} 
\numberwithin{equation}{section}

\allowdisplaybreaks

\newtheorem{theo}{Theorem}[section]

\newtheorem{lemm}[theo]{Lemma}
\newtheorem{prop}[theo]{Proposition}

\theoremstyle{definition}
\newtheorem{defi}[theo]{Definition}

\newtheorem{rema}[theo]{Remark}

\newtheorem{assume}[theo]{Assumption}

\newcommand{\D}{\mathbb D}
\newcommand{\E}{\mathbb E}
\newcommand{\F}{\mathbb F}
\newcommand{\bN}{\mathbb N}
\renewcommand{\P}{\mathbb P}

\newcommand{\R}{\mathbb R}

\newcommand{\1}{\mathbbm{1}}

\newcommand{\od}{\mathrm{d}}

\newcommand{\cB}{\mathcal B}

\newcommand{\cF}{\mathcal F}
\newcommand{\cG}{\mathcal G}

\newcommand{\cL}{\mathcal L}
\newcommand{\cN}{\mathcal N}
\newcommand{\cP}{\mathcal P}
\newcommand{\cR} {{\mathcal R}}

\newcommand{\cC}{\mathcal C}

\newcommand{\vz}{\varphi}
\newcommand{\Oz}{\Omega}
\newcommand{\oz}{\omega}
\newcommand{\tz}{\theta}
\newcommand{\Tz}{\Theta}

\newcommand{\dz}{\delta}
\newcommand{\Dz}{\Delta}

\newcommand{\sz}{\sigma}
\newcommand{\az}{\alpha}
\newcommand{\bz}{\beta}

\newcommand{\lz}{\lambda}

\newcommand{\wt}{\widetilde}
\newcommand{\ol}{\overline}

\newcommand{\fz}{\infty}
\newcommand{\chf}{\mathds{1}}

\newcommand{\bigslant}[2]{{\raisebox{.2em}{$#1$}\left/\raisebox{-.2em}{$#2$}\right.}}

\newtheorem{assumption}{Assumption}
\newtheorem{assumptionalt}{Assumption}[assumption]
\newenvironment{assump}[1]{
  \renewcommand\theassumptionalt{#1}
  \assumptionalt
}{\endassumptionalt}

\begin{document}

\title[FBSDEs]{Coupling of forward-backward stochastic differential equations on the Wiener space, and application on regularity}
\author[X. Zhou]{Xilin Zhou}
\address{X. Zhou, Department of Mathematics and Statistics, University of Jyväskylä, P.O. Box 35 (MaD), FI-40014, Finland}
\email{xilin.j.zhou@jyu.fi}


\begin{abstract}
    S.~Geiss and J.~Ylinen proposed the coupling method \cite{Geiss:Ylinen:21} to investigate the regularity for the solution to the backward stochastic differential equations with random coefficients.
    In this paper, we explore this method in setting for the forward-backward stochastic differential equation with random and Lipschitz coefficients,
    We obtain the regularity in time,
    and the Malliavin Sobolev $\D_{1,2}$ differentiability for the solution.
\end{abstract}

\maketitle
\tableofcontents

\newpage

\section{Introduction}

We investigate the regularity of solutions to multidimensional forward-backward stochastic differential equations (FBSDEs) of the form
\begin{align}\label{fbsde:equ}
    \begin{cases}
        X_s = \xi + \int^s_t b(u,X_u,Y_u,Z_u)\,\od u + \int^s_t \left[\sz(u,X_u,Y_u) + A(u,Z_u)\right]\,\od W_u;\\
        \ \\
        Y_s = g(X_T) + \int^T_s f(u,X_u,Y_u,Z_u)\,\od u - \int^T_s Z_u\,\od W_u,
      \end{cases} 
\end{align}
where the predictable coefficients $b$, $\sz$, $A$, $f$ and $g$ are driven by
a multidimensional Brownian motion $W$.
We impose the following conditions on these coefficients:
\begin{itemize}
    \item[(1)] $b$, $\sz$, $A$, $f$ and $g$ are uniformly Lipschitz continuous in $(x,y,z)$;
    \item[(2)] $A$ is linear in $z$.
\end{itemize}
A process triple $\Tz = (X,Y,Z)$ is said to be a solution to equation \eqref{fbsde:equ}
if it satisfies the following conditions:
\begin{itemize}
    \item[(1)] $X$, $Y$ and $Z$ are adapted processes satisfying \eqref{fbsde:equ} in the It\^o's sense;
    \item[(2)] The norm $\|\Tz\|_{2,*}$ is finite, where, for any $p\geq 1$,
        \begin{equation*}
            \|\Tz\|^p_{p,*} := \E\left[\left|X^*_T\right|^p + \left|Y^*_T\right|^p + \left(\int^T_t \left|Z_s\right|^2\,\od s\right)^{p/2}\right],
        \end{equation*}
        with
        \begin{equation*}
            A^*_T := \sup_{s\in [t,T]} \left|A_s\right|,\qquad A \in \{X, Y\}.
        \end{equation*}
\end{itemize}
A solution $(X,Y,Z)$ satisfying these conditions
is referred to as an adapted $\cL^2$-solution of \eqref{fbsde:equ}.
It is clear that the solution $X$ and $Y$ admit continuous modifications (see, for example, \cite[Theorem 5.2.1]{Oksendal:03}).
Moreover, by martingale representation theorem, the process $Z$ can be a predictable process.
In this paper, we always assume that $X$ and $Y$ are continuous, and that $Z$ is predictable.
Due to a technical reason, we additionally assume that the value of the process $Z$ at starting time is $0$, i.e. $Z_t \equiv 0$,
which can be done without a loss of generality.

FBSDEs naturally emerge in the study of stochastic control problems through variational calculus.
When an optimal control exists, the forward equation characterizes the evolution of the controlled state,
while the backward equation represents the associated costate,
capturing the necessary conditions for optimality.
This interplay between forward and backward components plays a fundamental role in the formulation and solution of stochastic control problems,
bridging the dynamics of the system with optimal decision-making principles, see \cite{Karoui:Peng:Quenez:97,Ma:Yong:07,Oksendal:Sulem:05,Pardoux:Peng:90,Yong:Zhou:99}.

Regarding the existence and uniqueness of solutions to (strongly) coupled FBSDEs,
which are the focus of this paper, the first solvability result on a sufficiently small-time interval was obtained by Antonelli \cite{Antonelli:93},
and he also constructed a counterexample demonstrating that strongly coupled FBSDEs might fail to admit solutions over arbitrary intervals of finite length.
When all coefficients of the FBSDEs are uniformly Lipschitz continuous, the solvability on small-time intervals can be readily established using the contraction principle; see, for instance, \cite{Ma:Yong:07,Zhang:17}.
To extend solvability from small-time intervals to arbitrary intervals, several classical approaches have been developed, including:
the Four-Step Scheme introduced in \cite{Ma:Protter:Yong:94},
the method of continuation developed in \cite{Hu:Peng:95,Yong:97},
and the approach via the decoupling-field structure, presented in \cite{Ma2015,Zhang:17}.
Additionally, \cite{Ankirchner:Formm:Wendt:22} considers a transference method for the solvability for FBSDEs.
The $\cL^p$-theory for FBSDEs on small-time intervals was established by \cite{Yong:20}, and also by \cite{Xie:Yu:20}
with different assumptions.

For SDEs, path regularity can sometimes be established through direct computation.
However, for backward stochastic differential equations (BSDEs) and forward-backward stochastic differential equations (FBSDEs),
the situation is more subtle due to the presence of the auxiliary process $Z$ and the associated measurability issues.
Assuming deterministic coefficients, \cite{Ma:Zhang:02} and subsequently \cite{Imkeller:Reis} studied the regularity of solutions to BSDEs with Lipschitz continuous generators.
For decoupled FBSDEs, relevant results were obtained in \cite{Lionnet:Reis:Szpruch:15}, while fully coupled FBSDEs were investigated in \cite{Fromm:Imkeller:13}.
More recently, C.~Reisinger et al.~\cite{Reisinger:Stockinger:Zhang:20} considered the case of coupled McKean--Vlasov FBSDEs.
Under similar assumptions, the Malliavin--Sobolev differentiability of the solution components has also been established in above articles.
Furthermore, a recent series of papers \cite{Imkeller:Pellat:Menoukeu-Pamen:24,Pellat:Che-Fonka:Menoukeu-Pamen:24} extended the analysis to cases with random coefficients,
investigating the Sobolev differentiability property under appropriate structural and regularity conditions.

To investigate the regularity properties for BSDE in a more general setting,
in \cite{Geiss:Ylinen:21}, 
Geiss and Ylinen introduced the coupling method.
Recently, Geiss and Zhou \cite{Geiss:Zhou:24} established a characterization of the space $\D_{1,2}$ via the coupling method
and derived regularity results SDEs and fully decoupled FBSDEs.
This naturally raises the question of how the coupling method can be applied to investigate the regularity of FBSDEs.
In this article, by using different kinds of coupling operator,
we characterize the regularity properties for the solution to the FBSDEs with random coefficients.
To achieve this target, we provide an estimate for the ``coupling variance'' defined in \eqref{equ:decouplingVariance} for any given coupling operator.
If the coefficients are further assumed to be deterministic, we also obtain some new results for this classical case.

\medskip

This article is organized as follows:

Section~\ref{sec:pre} provides the fundamental background for our study and introduces the notation used throughout the article.
\smallskip

Section~\ref{sec:transference} revisits the coupling method introduced in \cite{Geiss:Ylinen:21} and summarizes relevant results from \cite{Geiss:Ylinen:21} and \cite{Geiss:Zhou:24}.
Some details related to the coupling method are shown in the Appendix (Section~\ref{sec:Appendix:coupling}).
\smallskip

Section~\ref{sec:lp} recalls the $\cL^p$-theory on small-time intervals established by J.~Yong.
Moreover, with the help of the decoupling field structure,
we try to extend the local result to the $\cL^p$-theory on arbitrary finite length time intervals.
The idea follows directly from the $\cL^2$-theory in \cite{Zhang:17}.
\smallskip

Section~\ref{sec:couplingVariance} investigates
the estimates for the $p$-coupling variance $CV_p([t,T])$ for $p\geq 2$.
This target is closely related to the question about regularity considered in \cite{Geiss:Ylinen:21} and \cite{Geiss:Zhou:24}.
We first obtain the estimates when the time interval is sufficiently small.
After that, we investigate the transference of the decoupling field under the coupling operator,
and then extend our estimates for $CV_p([t,T])$ to arbitrary time interval of finite length.
\smallskip

Section~\ref{sec:regularity}, as an application of the results in Section~\ref{sec:couplingVariance},
establishes the path-regularity ($\cL^p$ regularity in time)
and $\D_{1,2}$ differentiability of solutions to FBSDEs.
\smallskip

\section{Settings and preliminaries}
\label{sec:pre}

\subsection{Basic Notations}
Throughout this article, we use the following notation:
\begin{itemize}
    \item[(1)] $n,m,d\in\bN^+$ are fixed constants representing dimensions.
    \item[(2)] For any given $A,B\ge 0$,
    we write $A \lesssim B$ if there exists a constant $c \geq 1$,
    referred to as the equivalence constant, such that 
    $B \le c A$. We denote $A\sim B$ if both $A \lesssim B$ and $B\lesssim A$ hold.
    \item[(3)] For any $n_1,n_2\in\bN^+$ and $x\in\R^{n_1\times n_2}$, 
    we denote by $|x|$ the euclidean norm of $x = (x_{i,k})_{i = n_1,k = n_2}$, given by
    \begin{equation*}
        |x|:= \sqrt{\sum_{i = 1,k = 1}^{n_1,n_2} |x_{i,k}|^2}.
    \end{equation*}
    \item[(4)] Let $\lz$ be the Lebesgue measure on a given time interval,
    and ``$\od u$" to represent integration with respect to $\lz$.
\end{itemize}
\smallskip

\subsection{Stochastic Basis}
We consider a stochastic basis satisfying the usual conditions.
Let $T \in (0,\fz)$ be a fixed time horizon, and define
\[     W = (W_s)_{s \in [0, T]}
     = \left ( (W_{s,1},\ldots,W_{s,d})^\top \right )_{s \in [0, T]} \]
as a $d$-dimensional standard Brownian motion on a complete probability space $(\Omega,\cF^W,\P)$,
where $W_0\equiv 0$ and all paths are continuous.
We denote by  $\F^W := (\cF^W_s)_{s\in [0,T]}$ the augmented natural filtration of $W$ and
by $\cF^W := \cF^W_T$ the terminal $\sz$-algebra.
When no other Brownian motion is involved, we omit the superscript $W$ for simplicity.
\smallskip

\subsection{Norm spaces}
The following function spaces will be used throughout this article.
For $p\in [1,\fz]$, $q\in [1,\fz)$ and $N\in\bN^+$:
\begin{itemize}
        \item $\cL^0(\Oz,\cF;\R^N)$: The space of all random vectors
        $$v:(\Oz,\cF)\to (\R^N,\cB(\R^N)).$$
        \item $\cL^0([0,T] \times\Oz,\F;\R^N)$: The space of all $\F$-progressively measurable stochastic processes
        $$X: ([0,T]\times\Oz, \cB([0,T])\otimes \cF) \to (\R^N,\cB(\R^N).$$
        \item $\cL^q(\Oz,\cF,\P;\R^N)$: The space of all $v\in \cL^0(\Oz,\cF;\R^N)$ satisfying
        $$\|v\|^q_{\cL^q}:= \E_{\P}[|v|^q]<\fz.$$
        \item $\cL^{p,q}([0,T]\times\Oz,\F,\P;\R^N)$: The space of all $w\in \cL^0([0,T]\times\Oz,\F;\R^N)$ satisfying
        $$\|w\|^q_{\cL^{p,q}} := \E_{\P}\left[\left(\int^T_0|w_s|^p\,\od s\right)^{q/p}\right]<\fz,\quad\text{if}\quad p \in (0,\fz),$$
        and
        $$\|w\|^q_{\cL^{\fz,q}} := \E_{\P}\left[\sup_{s\in [0,T]} |w_s|^q\right] < \fz,\quad\text{if}\quad p = \fz.$$
        \item $\cL^{\fz,q}_C([0,T]\times\Oz,\F,\P;\R^N)$: The space of all $w\in \cL^0([0,T]\times\Oz,\F;\R^N)$ satisfying:
        \begin{itemize}
            \item[(1)] $w$ is $\F$-adapted, and $s\to w(s,\oz)$ is continuous for all $\oz$;
            \item[(2)] $w\in \cL^{\fz,q}([0,T]\times\Oz,\F,\P;\R^N)$.
        \end{itemize}
\end{itemize}

Additionally, we consider the $L^p$ spaces as quotient spaces:
for all $p\in [1,\fz)\cup\{0\}$ and $q\in [1,\fz]$,
$$L^p(\Oz,\cF,\P;\R^N) := \bigslant{\cL^{p}(\Oz,\cF;\R^N)}{\mathcal{N}_{\P}(\Oz,\cF)},$$ 
where
$$\mathcal{N}_{\P} := \{v\in \cL^{0}(\Oz,\cF;\R^N): \P(v = 0) = 1\}.$$
\smallskip

\subsection{Malliavin Sobolev Space  $\D_{1,2}$}
For a comprehensive introduction to the Malliavin Sobolev space
$\D_{1,2}\subseteq \cL^2(\Omega,\cF,\P)$,
we refer the reader to \cite{Nualart:06}. The Malliavin derivative is defined as a mapping
$$D: \D_{1,2} \to \cL^{2,2}(\Oz,\F,\P)$$
and $\D_{1,2}$ is a Hilbert space equipped with the norm
$$\| \xi\|_{\D_{1,2}}^2:=  \| \xi\|_{\cL^2}^2 + \| D\xi \|_{\cL^{2,2}}^2.$$
For all $p\in [2,\infty)$ the subspaces $\D_{1,p} \subseteq \D_{1,2}$ are defined via the norm
$$\| \xi\|_{\D_{1,p}} := \left ( \| \xi\|_{\cL^p}^p + \| D\xi \|_{\cL^{p,p}}^p \right )^\frac{1}{p}.$$

As an example, consider the case where $d = 1$ (i.e., the Brownian motion $W$ is one-dimensional).
Let $h:\R^K\to \R$ be a smooth function with compact support, and let
$$0 = t_0 < t_1 < \cdots < t_K = T.$$
If we define $\xi$ as
$$\xi=h(W_{t_1}-W_{t_0},\ldots, W_{t_K}-W_{t_{K-1}}),$$
then there is a representative of $D \xi$ given by
$$\sum_{k=1}^K \frac{\partial h}{\partial x_k}(W_{t_1}-W_{t_0},\ldots, W_{t_K}-W_{t_{K-1}}) \chf_{(t_{k-1},t_k]},$$
which belongs to $\cL^{2,2}(\Oz,\F,\P)$.
\smallskip

\subsection{The $\sigma$-algebra of predictable events}\label{sec:predictiable}
For a given $t\in [0,T)$,
let $\cP_{t,T}$ denote the predictable $\sz$-algebra
generated by the continuous $(\cF_s)_{s\in[t,T]}$-adapted processes
(see \cite[page 47]{Revuz:Yor:1999}).
In the special case $t = 0$, we write $\cP_T$ for simplicity.
\smallskip

\section{Coupling Method and Transference on the Wiener Space}
\label{sec:transference}

We recall the approach outlined in \cite[Chapter 3]{Geiss:Ylinen:21} and \cite[Section 3]{Geiss:Zhou:24}.

\subsection{Coupling Operator}\label{sec:coupling operator}
For each $i\in\{0,1\}$, consider a stochastic basis 
$$(\Omega^i,\cF^i,\P^i,(W^i_s)_{s\in[0,T]}),$$
where:
\begin{itemize}
    \item[(1)] $(\Omega^i,\cF^i,\P^i)$ is a complete probability space;
    \item[(2)] $W^i=(W^i_s)_{s\in[0,T]}$ is a $d$-dimensional Brownian motion with $W^i_0\equiv0$, and its paths are continuous.
\end{itemize}
Let $\F^i := (\cF^i_s)_{s\in[0,T]}$ denote the augmented natural filtration of $W^i$, and set $\cF^i:=\cF^i_T$.

Now, let $(R,\cR,\rho)$ be another probability space. \cite[Proposition 2.5(2)]{Geiss:Ylinen:21} shows that there exists a linear, isometric bijection
$$
\cC : L^0(R\times\Omega^0,\cR\otimes\cF^0,\rho\otimes\P^0;\R) \to L^0(R\times\Omega^1,\cR\otimes\cF^1,\rho\otimes\P^1;\R),
$$
where the metric on each space is defined by
$$
d_i(f,g) := \E_{\rho\otimes\P^i}\left[\frac{|f-g|}{1+|f-g|}\right].
$$

Informally, for any $\xi\in L^0(R\times\Omega^0,\cR\otimes\cF^0,\rho\otimes\P^0;\R)$,
it can be written into the form $\xi \sim f(r,W^0)$,
where $f$ is a map independent of the stochastic basis $(\Omega^0,\cF^0,\P^0)$.
The coupling operator can be sees as
$$\cC(\xi) \sim f(r,W^1).$$
The detailed construction will be shown in the Appendix~\ref{sec:Appendix:coupling}

By selecting different probability spaces $(R,\cR,\rho)$, we can specify two important cases:
\begin{align*}
\cC_0 &: L^0(\Omega^0,\cF^0,\P^0;\R) \to L^0(\Omega^1,\cF^1,\P^1;\R),\\[1mm]
\cC_T &: L^0\Bigl([0,T]\times\Omega^0,\cB([0,T])\otimes\cF^0,\frac{\lambda}{T}\otimes\P^0;\R\Bigr)\\[1mm]
    &\quad \to L^0\Bigl([0,T]\times\Omega^1,\cB([0,T])\otimes\cF^1,\frac{\lambda}{T}\otimes\P^1;\R\Bigr).
\end{align*}
This operator admits the following properties:

\begin{prop}[Proposition 2.5 and Lemma 2.8 in \cite{Geiss:Ylinen:21}]
\label{prop:2.5}
Let \(X, X_1, \dots, X_n \in L^0(R\times\Omega^0)\), and let \(Y_i \in \cC(X_i)\) for \(i=1,\dots,n\).
\begin{enumerate}
    \item The joint distribution is preserved; i.e. one has the following
    $$(Y_1,\dots,Y_n) \stackrel{d}{=} (X_1,\dots,X_n).$$
    \item For all Borel functions \(g:\R^n\to\R\),
    $$g(Y_1,\dots,Y_n) \in \cC\bigl(g(X_1,\dots,X_n)\bigr).$$
    \item For \(Y\in L^0([0,T]\times\Omega^1)\), we have \(Y\in \cC_T(X)\)
    if and only if there exists a null set \(N\subset [0,T]\) such that for all \(r\in [0,T]\setminus N\),
    $$Y(r,\cdot) \in \cC_0\bigl(X(r,\cdot)\bigr).$$
    \item Let $M$ be the complete metric space that is locally $\sz$-compact.
    If $f:[0,T]\times\Oz^0\times M \to \R$ satisfies
    \begin{itemize}
        \item $(r,\oz) \to f(r,\oz,x)$ is measurable for all $x\in M$;
        \item $x\to f(r,\oz,x)$ is continuous for all $(r,\oz)\in [0,T]\times\Oz^0$.
    \end{itemize}
    Then there is a function $g:[0,T]\times\Oz^1\times M \to \R$ satisfying all the conditions above, and $g(\cdot,x)\in \cC_0(f(\cdot,x))$ for all $x\in M$.
    Moreover, for all $\wt{g}$ which satisfies the same properties with $g$, it holds true that 
    $$\wt{g} = g, \quad\quad \frac{\lambda}{T}\otimes\P^0-a.s.$$
\end{enumerate}
\end{prop}

S.~Geiss and J.~Ylinen showed that
the operator $\cC_T$ maps a predictable process to a predictable process,
the case of progressively measurable process was not considered as not needed.
\begin{prop}[Proposition 2.12 in \cite{Geiss:Ylinen:21}]
    \label{prop:2.12}\,
    \begin{enumerate}
        \item[(1)] One has
        $$\cC_T\Bigl(L^0\bigl([0,T]\times\Omega^0, \cP^0_T\bigr)\Bigr) \subset L^0\bigl([0,T]\times\Omega^1, \cP^1_T\bigr),$$
        where $\cP^i_T$ denotes the predictable $\sigma$-algebra defined in Section \ref{sec:predictiable}.
        \item[(2)] If $X$ is a pathwise continuous and adapted process,
        then there exists a pathwise continuous adapted process $Y\in L^0([0,T]\times \Omega^1)$ such that
        $$Y(t) \in \cC_0\bigl(X(t)\bigr) \quad \text{for all } t\in [0,T].$$
    \end{enumerate}
    \end{prop}

    Moreover, after the transference under the coupling operator,
    the random functions can still keep their ``structures".
    \begin{prop}[Remark 3.4 in \cite{Geiss:Ylinen:21}]
        \label{prop:3.4}
        Assume that 
        \begin{itemize}
            \item $h^0:[0,T]\times\Omega^0 \to C(\R^N)$ is $(\cP^0,\cB(C(\R^N)))$-measurable;
            \item $h^1:[0,T]\times\Omega^1 \to C(\R^N)$ satisfies $h^1(x)\in \cC_T(h^0(x))$ for all $x\in\R^N$;
            \item $h^1$ is $(\cP^1,\cB(C(\R^N)))$-measurable.
        \end{itemize}
        Then the following holds:
        \begin{enumerate}[{(1)}]
        \item $h^0(\cdot,\cdot,0) \stackrel{d}{=} h^1(\cdot,\cdot,0)$ with respect to
              $\lambda\times\P^0$ and $\lambda\times\P^1$.
        \item Given a continuous $H:\R^N\times \R^N \to [0,\infty)$ such that, for all $(t,\omega^0,x_0,x_1)$,
              \[ |h^0(t,\omega^0,x_0) - h^0(t,\omega^0,x_1)| \le H(x_0,x_1), \]
             then we can choose $h^1$ such that,  for all $(t,\omega^1,x_0,x_1)$,
             \[ |h^1(t,\omega^1,x_0) - h^1(t,\omega^1,x_1)| \le H(x_0,x_1). \]
        \end{enumerate}
        \end{prop}

\subsection{A Special Setting}\label{sec:specialSetting}

In \cite{Geiss:Ylinen:21} and \cite{Geiss:Zhou:24}
the authors adopt a particular framework for the coupling operator that plays a crucial role in investigation for the regularity properties of stochastic differential equations.

For $d\ge 1$ and $T>0$, let 
$$
W = (W_t)_{t\in[0,T]} \quad \text{and} \quad W'=(W'_t)_{t\in[0,T]}
$$
be two $d$-dimensional Brownian motions with continuous paths and $W_0 = W'_0 \equiv 0$,
defined respectively on complete probability spaces $(\Omega,\cF,\P)$ and $(\Omega',\cF',\P')$,
where $\cF$ and $\cF'$ are the completions of $\sigma(W_t: t\in[0,T])$ and $\sigma(W'_t: t\in[0,T])$. 

Extend these processes to the product space
$$
\overline{\Omega}:=\Omega\times\Omega',\quad \overline{\P}:=\P\times\P',\quad \overline{\cF}=\ol{\cF\otimes\cF'}^{\overline{\P}},
$$
i.e. $W(\oz,\oz') := W(\oz)$ and $W'(\oz,\oz') := W'(\oz')$.

Given a measurable function $\varphi: [0,T]\to [0,1]$, by Lévy's characterization theorem the process
$$
W^{\varphi}_t:=\int_0^t\sqrt{1-\varphi(u)^2}\,\od W_u + \int_0^t\varphi(u)\,\od W'_u,\quad t\in[0,T],
$$
defines a $d$-dimensional Brownian motion on $(\overline{\Omega},\overline{\cF},\overline{\P})$. We denote its $\overline{\P}$-augmented natural filtration by
$$
\cF^{\varphi}_t:=\sigma\bigl(W^{\varphi}_s: s\in[0,T]\bigr)\vee\overline{\cN},
$$
where $\overline{\cN}$ is the collection of $\overline{\P}$-null sets.

Following the construction in Section~\ref{sec:coupling operator},
we now set
$$
(\Omega^0,\cF^0,\P^0,W^0):=(\overline{\Omega},\cF,\overline{\P},W),
$$
and
$$
(\Omega^1,\cF^1,\P^1,W^1):=(\overline{\Omega},\cF^{\varphi},\overline{\P},W^{\varphi}).
$$
The corresponding coupling operators $\cC_0$ and $\cC_T$,
as defined in Section~\ref{sec:coupling operator},
will now be denoted collectively by $\cC_{\varphi}$.
In particular, if the function $\vz$ is an indicator function on $[0,T]$,
i.e. $\varphi(u)=\1_{(a,c]}(u)$ for some $0 \leq a < c \leq T$, we simply write 
$$\cC_{(a,c]}:=\cC_{\varphi}.$$
Also, we adopt the following notational conventions:
\begin{itemize}
    \item[(1)] For any stochastic process \(X\), we denote by \(X^{\varphi}_u\) an element of the equivalence class \(\cC_{\varphi}(X)_u\).
    \item[(2)] For any random variable \(\xi\), we write \(\xi^{\varphi} \in \cC_{\varphi}(\xi)\).
    \item[(3)] For any function \(h:[0,T]\times\Omega\times\R^N\to\R^M\) satisfying
    \begin{itemize}
        \item $h$ is continuous in $x\in\R^N$;
        \item $h(\cdot,x)$ is measurable on $[0,T]\times\Omega$ for any $x\in\R^N$,
    \end{itemize}
    we denote by \(h^{\varphi}\) with the corresponding properties such that, for every \(x\in\R^N\), one has 
    $$h^{\varphi}(\cdot,x) \in \cC_{\varphi}\bigl(h(\cdot,x)\bigr).$$
\end{itemize}

\smallskip

The transference theorem for SDE (see Appendix \ref{sec:Appendix:trans}) shows that, if the following stochastic differential equation
$$
X_s = \xi + \int^s_t b(u,X_u)\,\od u + \int^t_s \sz(u,X_u)\,\od W_u
$$
has an $\cL^2$-solution, and the random coefficients $b$ and $\sz$ satisfies
some proper measurability conditions, then $X^{\vz}$ is an $\cL^2$-solution to the equation
$$
X^{\vz}_s = \xi^{\vz} + \int^s_t b^{\vz}(u,X^{\vz}_u)\,\od u + \int^t_s \sz^{\vz}(u,X^{\vz}_u)\,\od W^{\vz}_u.
$$
Moreover, by Proposition \ref{prop:3.4}, $(b^{\vz},\sz^{\vz})$ may share the same regularity with $(b,\sz)$.

\smallskip

S.~Geiss and J.~Ylinen established the following key estimate,
which helps with the estimation for the path-regularity of the solution to BSDEs.

\begin{theo}[Lemma 4.23 in \cite{Geiss:Ylinen:21}]\label{theo:path-regularity:BSDE}
    For all $0 \le a < c \le T$, consider the $\sigma$-algebra
    $$
    \cG^c_a := \sigma\bigl(W_t,\,t\in[0,a]\bigr) \vee \sigma\bigl(W_t - W_c,\,t\in[c,T]\bigr).
    $$
    Set $\vz = \1_{(a,c]}$. Then, for all $p\in [1,\infty]$ and $\xi\in \cL^p(\Omega,\cF,\P)$, one has
    $$
    \frac{1}{2^p}\,\E_{\ol{\P}}\Bigl[\bigl|\xi-\xi^{\vz}\bigr|^p\Bigr] \le \E_{\ol{\P}}\Bigl[\Bigl|\xi-\E\bigl[\xi\mid\cG^c_a\bigr]\Bigr|^p\Bigr] \le \E_{\ol{\P}}\Bigl[\bigl|\xi-\xi^{\vz}\bigr|^p\Bigr].
    $$
\end{theo}

Recently, S.~Geiss and X.~Zhou provided a characterization of the Malliavin Sobolev space $\D_{1,2}$
via the coupling operator. The result is as follows:

\begin{theo}[Corollary 4.3 in \cite{Geiss:Zhou:24}]\label{theo:Malliavin_sobolev}
    For all $\xi\in \cL^2(\Omega,\cF,\P)$, one has $\xi\in \D_{1,2}$ if and only if
    $$
    \sup_{r\in (0,1]} \frac{\E_{\ol{\P}}\Bigl[\bigl|\xi-\xi^{\vz_r}\bigr|^2\Bigr]}{r^2} < \infty,
    $$
    where for each $r\in (0,1]$, $\varphi_r$ is a constant function with value $r$.
\end{theo}

In this article, we focus on estimating a special structure,
which we refer to as the $p$-coupling variance. On the time interval $[t,T]$, for $t \leq t_1 < t_2\leq  T$, we define:
\begin{align}\label{equ:decouplingVariance}
    CV_p([t_1,t_2]) := &\E\left[\sup_{u\in [t_1,t_2]} \Bigl|X_u - X^{\vz}_u\Bigr|^p + \sup_{u\in [t_1,t_2]} \Bigl|Y_u - Y^{\vz}_u\Bigr|^p\right.\notag \\
    &\quad\quad + \left(\int_{t_1}^{t_2} \Bigl|Z_u - Z^{\vz}_u\Bigr|^2\,du\right)^{p/2}\notag\\
    &\quad\quad \left.+ \left(\int_{t_1}^{t_2} \Bigl(1-\sqrt{1-(\varphi(u))^2}\Bigr)|Z_u|^2\,du\right)^{p/2}\right],
\end{align}
where $(X,Y,Z)$ is a solution to an FBSDE.

By Theorem~\ref{theo:path-regularity:BSDE} and Theorem~\ref{theo:Malliavin_sobolev},
the first two terms of $CV_p([t,T])$ offer the information for the regularity of the processes $X$ and $Y$.
The final term in $CV_p([t,T])$ naturally emerges in the analysis and plays a crucial role in establishing path-regularity results for the process $Z$.
The primary objective of this article is to analyze and estimate this term in the FBSDE setting.

\subsection{Technical Considerations Related to the Coupling Operator}

In this article, when applying the coupling operator within the framework described in Section~\ref{sec:specialSetting},
we begin by extending any random variable, stochastic process, or random function originally defined on $(\Omega,\cF,\P)$
to the canonical product space $(\overline{\Omega},\overline{\cF},\overline{\P})$.
Once this extension is established, we proceed with the coupling operator transformation.

In some cases, the SDE under consideration may not originate at time $0$.
However, the coupling operator remains well-defined and retains the same fundamental properties in such scenarios.
For a detailed technical justification, see \cite[Section 3]{Geiss:Zhou:24}.
Consequently, we do not distinguish this case explicitly in our analysis.

\section{$\cL^p$ Theory for FBSDEs}
\label{sec:lp}
In this section, we first recall the $\cL^p$ theory for forward-backward stochastic differential equations (FBSDEs) with random coefficients as developed in \cite{Yong:20}.
After stating the original theorem, we present several natural propositions that follow from it.
The original theorem is established under the assumption that the time interval is sufficiently small. By employing the decoupling field structure and standard techniques,
this $\cL^p$ theory can be extended to cover time intervals of arbitrary length.

We consider FBSDEs of the form
\begin{align}\label{fbsde:equ:general}
    \begin{cases}
        X_s = \xi + \displaystyle\int_t^s b(u,X_u,Y_u,Z_u)\,du + \displaystyle\int_t^s \mu(u,X_u,Y_u,Z_u)\,dW_u,\\[1mm]
        Y_s = g(X_T) + \displaystyle\int_s^T f(u,X_u,Y_u,Z_u)\,du - \displaystyle\int_s^T Z_u\,dW_u,
      \end{cases} 
\end{align}
where the random coefficients satisfy the following conditions:
\begin{assume}\label{coefficient:assume:general}\,
    \begin{enumerate}[(1)]
        \item \textbf{Domain:} The coefficient functions are defined on their appropriate spaces:
        \begin{align*}
            b: &\ [0,T]\times\Omega\times\R^n\times\R^m\times\R^{m\times d}\to \R^n,\\[1mm]
            \mu: &\ [0,T]\times\Omega\times\R^n\times\R^m\times\R^{m\times d}\to \R^{n\times d},\\[1mm]
            f: &\ [0,T]\times\Omega\times\R^n\times\R^m\times\R^{m\times d}\to \R^m,\\[1mm]
            g: &\ \Omega\times\R^n\to \R^m.
        \end{align*}
        \item \textbf{Measurability:} For all $\tz=(x,y,z)\in\R^n\times\R^m\times\R^{m\times d}$, the functions
        $$
        b(\cdot,\tz),\quad \mu(\cdot,\tz),\quad f(\cdot,\tz)
        $$
        are predictable on $[0,T]\times\Oz$, and for all $x\in\R^n$, $g(\cdot,x)$ is $\cF_T$-measurable.
        \item \textbf{Lipschitz Continuity:} There exist positive constants $L_{b,i}$, $L_{\mu,i}$, $L_{f,i}$ (for $i=1,2,3$) and $L_g$,
        such that for all $s\in[0,T]$, $\omega\in\Omega$, and all $\tz=(x,y,z)$, $\tz'=(x',y',z')\in\R^n\times\R^m\times\R^{m\times d}$,
        \begin{align*}
            |b(s,\omega,\tz)-b(s,\omega,\tz')| &\le L_{b,1}|x-x'| + L_{b,2}|y-y'| + L_{b,3}|z-z'|,\\[1mm]
            |\mu(s,\omega,\tz)-\mu(s,\omega,\tz')| &\le L_{\mu,1}|x-x'| + L_{\mu,2}|y-y'| + L_{\mu,3}|z-z'|,\\[1mm]
            |f(s,\omega,\tz)-f(s,\omega,\tz')| &\le L_{f,1}|x-x'| + L_{f,2}|y-y'| + L_{f,3}|z-z'|,\\[1mm]
            |g(\omega,x)-g(\omega,x')| &\le L_g|x-x'|.
        \end{align*}
    \end{enumerate}
\end{assume}

Moreover, for a given constant $p\in[2,\infty)$, we impose the following integrability conditions on the coefficients:

\begin{assump}{H(p)}\label{fbsde:basic:assume:p:general}
    For each $s\in[0,T]$, define
    \begin{equation}\label{equ:def:0function}
    h^0(s,\omega) := h(s,\omega,0,0,0),\quad \text{for } h\in\{b,f,\mu\}.
    \end{equation}
    Then, the following conditions are assumed: for the given $t\in [0,T)$ (the starting time for the equation)
    \begin{enumerate}[(1)]
        \item the function $b^0$ belongs to $\cL^{1,p}([t,T]\times\Omega,\F,\P;\R^n)$;
        \item the function $f^0$ belongs to $\cL^{1,p}([t,T]\times\Omega,\F,\P;\R^m)$;
        \item the function $\mu^0$ belongs to $\cL^{2,p}([t,T]\times\Omega,\F,\P;\R^{n\times d})$;
        \item the function $g(0)$ belongs to $\cL^p(\Omega,\cF_T,\P;\R^m)$;
        \item the initial condition $\xi$ belongs to $\cL^p(\Omega,\cF_t,\P;\R^n)$.
    \end{enumerate}
\end{assump}

\subsection{\(\cL^p\) Theory on a Small Time Interval}

We begin by recalling the Burkholder--Davis--Gundy (BDG) inequalities.
The following lemma ensures that for any $p\in (0,\infty)$
there exist constants $0 < \underline{K}_p < \overline{K}_p$
(depending only on $p$) such that for any $G \in \cL^{2,p}([0,T]\times\Omega,\F,\P)$
and for every $t\in[0,T]$, one has the following statement.
\begin{lemm}\label{BDG:p}
There exist constants $0<\underline{K}_p<\overline{K}_p<\fz$, depending only on $p$, such that, a.s.,
\begin{align*}
    \underline{K}_p \, \E\Biggl[\Bigl(\int_t^T |G_s|^2\,ds\Bigr)^{\frac{p}{2}} \Bigg| \cF_t\Biggr] 
    &\le \E\Biggl[\sup_{s\in[t,T]}\Bigl|\int_t^s G_u\,dW_u\Bigr|^p \Bigg| \cF_t\Biggr] \\
    &\le \overline{K}_p \, \E\Biggl[\Bigl(\int_t^T |G_s|^2\,ds\Bigr)^{\frac{p}{2}} \Bigg| \cF_t\Biggr].
\end{align*}
The constants \(\underline{K}_p\) and \(\overline{K}_p\) are independent of the dimension of $G$ and $W$.
\end{lemm}

J.~Yong showed the existence and uniqueness of the $\cL^p$-solution for \eqref{fbsde:equ:general}.
\begin{lemm}[Theorem 2.3 in \cite{Yong:20}]\label{p-exist}
Assume Assumption \ref{coefficient:assume:general}. For some $p\in[2,\infty)$, if
\begin{equation}\label{p-condition}
    \begin{cases}
        (\overline{K}_p)^{1/p}\left(\dfrac{p}{p-1} + 2(\underline{K}_p)^{-1/p}\dfrac{2p-1}{p-1}\right)L_g L_{\mu,3} < 1, & \text{if } p\in (2,\infty),\\[1mm]
        L_g L_{\mu,3} < 1, & \text{if } p=2,
    \end{cases}
\end{equation}
then there exists a constant \(\delta > 0\), depending on all the Lipschitz constants, such that if \(T-t\le \delta\), the FBSDE \eqref{fbsde:equ:general} admits a unique solution \(\Tz=(X,Y,Z)\) in \(\cL^p\)
where $X$ and $Y$ are continuous processes and $Z$ is predictable.
\end{lemm}

The proof of Lemma~\ref{p-exist} naturally leads to the following estimates for the solution $\Tz=(X,Y,Z)$,
including both the a priori estimate and the stability (variance) estimate.
Since the derivation follows standard techniques, 
we provide a detailed proof in Appendix~\ref{sec:Appendix:Lp}.
\begin{lemm}\label{lemm:p-estimate}
    Let $\Tz := (X,Y,Z)$ be the unique solution to the FBSDEs satisfying all the conditions in Lemma~\ref{p-exist}.
    Then, the following a priori estimate holds:
    \begin{align}\label{p-estimate}
        \|\Tz\|^p_{p,*} \leq c_{\eqref{p-estimate}}
        &\E\Biggl[|\xi|^p + |g(0)|^p + \Biggl(\int_t^T |b^0(s)|\,\od s\Biggr)^p \notag\\
        &\quad + \Biggl(\int_t^T |\mu^0(s)|^2\,\od s\Biggr)^{p/2} + \Biggl(\int_t^T |f^0(s)|\,\od s\Biggr)^p\Biggr] < \infty,
    \end{align}
    where the constant $c_{\eqref{p-estimate}}$ depends only on $p$, $\dz$, and all the Lipschitz constants.

    Furthermore, consider another set of coefficients
    $(\wt{\xi},\wt{b},\wt{\mu},\wt{f},\wt{g})$ that satisfies all the conditions in Lemma~\ref{p-exist},
    and shares the same Lipschitz constants as $(\xi,b,\mu,f,g)$.
    Let $\wt{\Tz} = (\wt{X},\wt{Y},\wt{Z})$ be the corresponding solution to the modified FBSDEs.
    Then, the following stability estimate holds:
    \begin{align}\label{p-var:equ}
        \|\wt{\Tz} - \Tz\|^p_{p,*}
        \leq c_{\eqref{p-var:equ}} 
        &\E\Biggl[|\xi - \wt{\xi}|^p + |g(X_T) - \wt{g}(X_T)|^p + V_p([t,T])\Biggr],
    \end{align}
    where
    \begin{align*}
        V_p([t,T]) 
        &:= \Biggl(\int_t^T |\Delta b(s,\Tz_s)|\,\od s\Biggr)^p 
        + \Biggl(\int_t^T |\Delta \mu(s,\Tz_s)|^2\,\od s\Biggr)^{p/2}\\
        &\quad\quad + \Biggl(\int_t^T |\Delta f(s,\Tz_s)|\,\od s\Biggr)^p,
    \end{align*}
    with
    \begin{equation*}
        \Delta h := h - \wt{h}, \quad h\in\{b,\mu,f\}.
    \end{equation*}
    The constant $c_{\eqref{p-var:equ}}$ depends only on $p$, $\dz$, and all the Lipschitz constants.
\end{lemm}

\subsection{Decoupling Field}\label{sec:decouplingField}
For the $\cL^2$-theory of FBSDEs, the decoupling field structure provides a method to extend the time interval to an arbitrary length (see \cite{Zhang:17} and \cite{Ma2015} for a complete view). 
In this section, we aim to investigate the $\cL^p$-theory for FBSDEs over time intervals of arbitrary length.
\begin{defi}\label{def:dfield}
    An $\F$-progressively measurable random field 
    $w:[0,T]\times \Oz\times \R^n \to \R^m$ with $w(T,x) = g(x)$ is called a 
    ``decoupling field" of FBSDE \eqref{fbsde:equ:general} if there is a constant $\dz > 0$ 
    such that, for all $ 0 \leq t_1 < t_2 \leq T$ with $t_2 - t_1 \leq \dz$,
    and all $\xi \in \cL^2(\Oz,\cF_{t_1},\P;\R^N)$, the FBSDE
    \begin{align}\label{fbsde:equ:de}
        \begin{cases}
            X_s = \xi + \int^s_{t_1} b(u,X_u,Y_u,Z_u)\,\od u + \int^s_{t_1} \mu(u,X_u,Y_u,Z_u)\,\od W_u;\\
            \ \\
            Y_s = w(t_2,X_{t_2}) + \int^{t_2}_s f(u,X_u,Y_u,Z_u)\,\od u - \int^{t_2}_s Z_u\,\od W_u,
          \end{cases} 
    \end{align}
    has a unique solution, denoted by $\Tz^{\xi,[t_1,t_2]} := (X^{\xi,[t_1,t_2]},Y^{\xi,[t_1,t_2]},Z^{\xi,[t_1,t_2]})$,
    that satisfies, for all $s\in [t_1, t_2]$,
    \begin{equation}\label{dfield:property}
        Y^{\xi,[t_1,t_2]}_s = w(s,X^{\xi,[t_1,t_2]}_s)\quad \P-a.s.
    \end{equation}
\end{defi}
\begin{defi}
    A decoupling field $w$ is called regular 
    if it is uniformly Lipschitz in $x$, 
    i.e. there is a positive constant $L_w > 0$,
    for all $x,y\in\R^n$ and $(s,\oz)\in [0,T]\times\Oz$,
    \begin{equation}\label{dfield:regular}
        |w(s,\oz,x) - w(s,\oz,y)|\leq L_w|x - y|.
    \end{equation}
\end{defi}

\smallskip

The following existence and uniqueness theorem holds.
\begin{lemm}\label{lemma:p-exits:decoupling}
    For arbitrary $0 \leq t < T < \fz$, suppose the FBSDE \eqref{fbsde:equ:general} satisfies all the following conditions:
    \begin{itemize}
        \item[(1)] satisfies Assumption \ref{coefficient:assume:general} and Assumption \ref{fbsde:basic:assume:p:general} with some $p\in [2,\fz)$;
        \item[(2)] admits a regular decoupling field $w$ with the Lipschitz constant $L_w$;
        \item[(3)] \eqref{p-condition} holds with $L_g$ replaced by $L_w$. 
    \end{itemize}
    Then, the equation \eqref{fbsde:equ:general} admits a unique $\cL^p$-solution $\Tz = (X,Y,Z)$
    such that the a priori estimate \eqref{p-estimate} holds on $[t,T]$.
\end{lemm}
\begin{proof}
    The proof follows a similar approach as in the $\cL^2$-case (see, for example, the proof of \cite[Theorem 7.3]{Ma2015}):
    First we use Lemma~\ref{p-exist} to obtain the estimate for \eqref{p-estimate} on sub-intervals with sufficiently small length. 
    Then \eqref{dfield:property} guarantees that the estimates can be extended to the full time interval $[t,T]$.
    We omit the details here.
\end{proof}
\begin{rema}
    In \cite{Ma2015}, Ma et al. provided a comprehensive analysis of the existence of regular decoupling fields
    and established connections between the coefficients and the Lipschitz constant of the decoupling field.
    Below, we summarize relevant results that can be used to verify the assumptions in Lemma~\ref{lemma:p-exits:decoupling}:
    \begin{itemize}
        \item[(1)] When $\mu = \mu(s,x,y,z)$, we refer to \cite[Theorem 7.4]{Ma2015};
        \item[(2)] When $\mu = \mu(s,x,y)$, we refer to \cite[Theorem 7.3]{Ma2015};
        \item[(2')] In the one-dimensional case, i.e., $n = m = d = 1$
        with $\mu = \mu(s,x,y)$, a detailed discussion is available in \cite[Theorem 8.3.5]{Zhang:17}.
    \end{itemize}
\end{rema}

\section{Coupling Variance Estimation}\label{sec:couplingVariance}

For a given measurable function $\vz:[0,T]\to [0,1]$,
and the corresponding coupling operator $\cC_{\vz}$ defined in Section~\ref{sec:specialSetting}.
We now present the main theorem of this article.
Our goal is to establish an estimate for the $p$-coupling variance $CV_p([t,T])$ with respect to $\cC_{\vz}$
for the FBSDE \eqref{fbsde:equ}:
\begin{align}\label{fbsde:equ:re}
    \begin{cases}
        X_s = \xi + \int^s_t b(u,X_u,Y_u,Z_u)\,\od u + \int^s_t \sz(u,X_u,Y_u) + A(u,Z_u)\,\od W_u;\\
        \ \\
        Y_s = g(X_T) + \int^T_s f(u,X_u,Y_u,Z_u)\,\od u - \int^T_s Z_u\,\od W_u,
      \end{cases} 
\end{align}
which is a particular case of the general FBSDE formulation in \eqref{fbsde:equ:general},
under the following assumption:
\begin{assume}\label{ass:linear:diffusion}
    The coefficient $\mu$ admits the following form:
    $$\mu(u,\oz,x,y,z) = \sz(u,\oz,x,y) + A(u,\oz,z),$$
    where 
    \begin{itemize}
        \item $\sz: [0,T]\times\Omega\times\R^n\times\R^m\to \R^{n\times d}$ is predictable;
        \item $A: [0,T]\times \Oz\times \R^{m\times d}\to \R^{n\times d}$ is predictable and linear in $\R^{m\times d}$.
    \end{itemize}
\end{assume}
\begin{rema}
    Notice that, if $\mu$ is uniformly Lipschitz in $(x,y,z)$, then the function $\sz$ is uniformly Lipschitz in $(x,y)$,
    and $A$ is uniformly Lipschitz in $z$.
\end{rema}

\subsection{On Small Time Intervals}\label{sec:small}
When the time interval $T - t$ is sufficiently small,
the estimation of the $p$-coupling variance $CV_p([t,T])$
can be directly obtained using Lemma~\ref{lemm:p-estimate} along with straightforward calculations.
In the remainder of this section,
we rigorously present the detailed steps of the proof,
outlining how the estimate follows from the fundamental properties of the FBSDE system and the coupling operator.

\smallskip

By Assumption \ref{ass:linear:diffusion},
the following technic lemma holds by direct calculation.
\begin{lemm}\label{lemma:vz:sz}
    Suppose that $\mu$ in \eqref{fbsde:equ:general}
    satisfies Assumption \ref{coefficient:assume:general}, Assumption \ref{fbsde:basic:assume:p:general},
    and is decomposed as $\mu = \sz + A$ according to Assumption \ref{ass:linear:diffusion}. 
    For all $\az\in [0,1]$ and 
    $$(s,\oz,x,y,(z_1,z_2))\in[0,T]\times \Oz\times \R^n\times\R^m \times(\R^{m\times d})^2,$$
    set
    \begin{equation}\label{defi:Sigma}
        \Sigma(s,\az,\oz,x,y,(z_1,z_2)) 
        := \left(\sz_{\sqrt{1 - \az^2}}(s,\oz,x,y,z_1), \sz_{\az}(s,\oz,x,y,z_2)\right),
    \end{equation}
    where
    $\sz_{r}(s,\oz,x,y,z) := r\,\sz(s,\oz,x,y) + A(s,\oz,z)$ for all $r\in [0,1]$.
    Then the following properties hold:
    \begin{itemize}
        \item[(1)] Lipschitz Continuity:
        For all $\az\in [0,1]$, $s\in [0,T]$, $\oz\in\Oz$ and 
        $$\tz = (x,y,(z_1,z_2)),\quad \tz' = (x',y',(z'_1,z'_2)),$$
        we have
        \begin{align*}
            |\Sigma(s,\az,\tz) - \Sigma(s,\az,\tz')| 
            &\leq L_{\mu,1} |x - x'| + L_{\mu,2}|y - y'|\\
            &\quad\quad + L_{\mu,3} |(z_1,z_2) - (z'_1,z'_2)|,
        \end{align*}
        where $\{L_{\mu,i}\}^3_{i = 1}$ are the Lipschitz constants for $\mu$ (Assumption \ref{coefficient:assume:general});
        \item[(2)] Integrability Condition:
        The function $\Sigma(\cdot,0,0,(0,0))$ belongs to the space $\cL^{2,p}([t,T]\times\Oz,\F,\P;\R^{n\times 2d})$;
        \item[(3)] Representation Property:
        For all $z\in \R^{m\times d}$,
        \begin{align*}
            &\Sigma(s,\az,x,y,(\sqrt{1 - \az^2}z,\az z))\\
            &\quad = \left(\sqrt{1 - \az^2} \mu(s,x,y,z), \az\mu(s,x,y,z)\right).
        \end{align*}
    \end{itemize}
\end{lemm}

\smallskip

If the FBSDE \eqref{fbsde:equ:re}
admits an $\cL^2$ solution $\Tz = (X,Y,Z)$,
the coupling operator $\cC_{\vz}$ can deduce a new FBSDE (Theorem \ref{theo:transference}):
\begin{align}\label{fbsde:decouple:origin:0}
    \begin{cases}
    X^{\vz}_s = \xi^{\vz} + \int^s_t b^{\vz}(u,X^{\vz}_u,Y^{\vz}_u,Z^{\vz}_u)\,\od u + \int^s_t \mu^{\vz}(u,X^{\vz}_u,Y^{\vz}_u,Z^{\vz}_u)\,\od W^{\vz}_u;\\
    \ \\
    Y^{\vz}_s = g^{\vz}(X^{\vz}_T) + \int^T_s f^{\vz}(u,X^{\vz}_u,Y^{\vz}_u,Z^{\vz}_u)\,\od u - \int^T_s Z^{\vz}_u\,\od W^{\vz}_u.
    \end{cases}
\end{align}
By Lemma~\ref{lemma:vz:sz}(3) and the fact that
$$\od W^{\vz}_u = \sqrt{1 - (\vz(u))^2}\,\od W_u + \vz(u)\,\od W'_u,$$
we can rewrite the FBSDE \eqref{fbsde:decouple:origin:0} in the form:
\begin{align}\label{fbsde:decouple:origin}
    \begin{cases}
        \begin{aligned}
            X^{\vz}_s &= \xi^{\vz} 
            + \int^s_t b^{\vz}(u,X^{\vz}_u,Y^{\vz}_u,Z^{\vz}_u)\,\od u + \int^s_t \ol{\sz}(u)\,\od \ol{W}_u;
        \end{aligned} \\
        \ \\
        \begin{aligned}
            Y^{\vz}_s &= g^{\vz}(X^{\vz}_T) 
            + \int^T_s f^{\vz}(u,X^{\vz}_u,Y^{\vz}_u,Z^{\vz}_u)\,\od u \\
            &\quad\quad - \int^T_s \bigl(\sqrt{1 - (\vz(u))^2}Z^{\vz}_u, \vz(u)Z^{\vz}_u\bigr)\,\od \ol{W}_u,
        \end{aligned}
    \end{cases}
\end{align}
where 
$\ol{W} := \begin{pmatrix}
    W \\
    W'
  \end{pmatrix}$ 
is a $2d$-dimensional Brownian motion, and
\begin{align*}
    \ol{\sz}(u) 
    & = \Sigma^{\vz}(u,\vz(u),X^{\vz}_u,Y^{\vz}_u, (\sqrt{1 - (\vz(u))^2}Z^{\vz}_u, \vz(u)Z^{\vz}_u)).
\end{align*}

Then, we can redefine the other coefficients (i.e. $b^{\vz},f^{\vz}$ in \eqref{fbsde:decouple:origin})
in the $z$-component such that 
$$(X^{\vz},Y^{\vz},(\sqrt{1 - (\vz(\cdot))^2}Z^{\vz}, \vz(\cdot)Z^{\vz}))$$
become the solution for a new FBSDE with respect to the Brownian motion $\ol{W}$.

\begin{rema}
By  Lemma~\ref{p-exist}, the Lipschitz constants for the coefficients in the diffusion part, i.e. $\mu$ and $\mu^{\vz}$,
and the terminal function i.e. $g$ and $g^{\vz}$, are essential for the existence of the solution to FBSDE.
By Lemma \ref{lemma:vz:sz}(1) and Proposition \ref{prop:3.4},
$(\Sigma^{\vz},g^{\vz})$ admits the same Lipschitz constants as $(\mu,g)$.
Thus, the new FBSDEs derived from \eqref{fbsde:decouple:origin} will have the similar solvability as the original FBSDE \eqref{fbsde:equ:re}.
\end{rema}

Now we present the main theorem in this section,
which focus on the estimation for coupling variance on small-time interval.
\begin{theo}\label{thm:small:decoupling}
    Let Assumption~\ref{coefficient:assume:general}, Assumption~\ref{fbsde:basic:assume:p:general},
    Assumption~\ref{ass:linear:diffusion} and \eqref{p-condition} hold for some $p\in [2,\fz)$. 
    Then there is $\dz > 0$ such that \eqref{fbsde:equ:re} admits a unique solution
    $\Tz := (X,Y,Z)$ on $[t,T]$ if $T - t \leq \dz$.
    Moreover, if $T - t \leq \dz$,
    considering the coupling operator $\cC_{\vz}$ with respect to the measurable function $\vz: [0,T] \to [0,1]$,
    the following estimation holds
    \begin{align}\label{equ:small:decoupling}
        CV_p([t,T])
        &\leq c_{\eqref{equ:small:decoupling}} \left\{\E\left[|\xi - \xi^{\vz}|^p + \left|g(X_T) - g^{\vz}(X_T)\right|^p\right] \right.\notag\\
        &\quad \qquad \left. + \E\left[U_p([t,T])\right] + S_p([t,T])\right\},
    \end{align}
    where the constant $c_{\eqref{equ:small:decoupling}} \in (1,\fz)$ only relies on $p$, $T - t$
    and all the Lipschitz constants in Assumption \ref{coefficient:assume:general}.
    Here the $U_p([t,T])$ is the potential of the coefficients and $S_p([t,T])$ is the term related to the function $\vz$, i.e.
    \begin{align*}
        U_p([t,T]) 
        &:= \left(\int^T_t |\Dz b(u,X_u,Y_u,Z_u)|\,\od u\right)^p\\
        &\quad\quad + \left(\int^T_t |\Dz \sz(u,X_u,Y_u) + \Dz A(u,Z_u)|^2\,\od u\right)^{p/2}\\
        &\quad\quad + \left(\int^T_t |\Dz f(u,X_u,Y_u,Z_u)|\,\od u\right)^p,
    \end{align*}
    with
    \begin{equation*}
        \Dz h := h - h^{\vz}\quad h\in\{b,\sz,A,f\};
    \end{equation*}
    and
    \begin{align*}
        S_p([t,T])
        &:= \left(\int^T_t (\vz(u))^2\,\od u\right)^{p/2}\\
        &\qquad \times \E\left[(L_{\mu,1})^p\sup_{u\in [t,T]}|X_u|^p + (L_{\mu,2})^p\sup_{u\in [t,T]}|Y_u|^p\right]\\
        &\qquad + \E\left[\left(\int^T_t (\vz(u))^2(\sz^0(u))^2\,\od u\right)^{p/2}\right].
    \end{align*}
\end{theo}

\begin{proof}
By Lemma \ref{p-exist}, there is $\dz_0 > 0$ such that \eqref{fbsde:equ} admits a unique $\cL^p$-solution $(X,Y,Z)$ when $T - t\leq \dz_0$.

Consider the $2d$-dimensional Brownian motion 
$\ol{W} := \begin{pmatrix}
    W \\
    W'
  \end{pmatrix}$
where $W$ and $W'$ be the Brownian motion
in the setting of Section~\ref{sec:specialSetting}.
The random coefficients $(b^{\vz},\sz^{\vz},A^{\vz},f^{\vz},g^{\vz})$
are defined as Section~\ref{sec:specialSetting} 
with original coefficients $(b,\sz,A,f,g)$ and coupling operator $\cC_{\vz}$.
By Proposition~\ref{prop:3.4} and Assumption~\ref{coefficient:assume:general},
we can choose $(b^{\vz},\sz^{\vz},A^{\vz},f^{\vz},g^{\vz})$ being uniformly Lipschitz
with same Lipschitz constants as $(b,\sz,A,f,g)$.
The corresponding Lipschitz constants denote by $L_{b,i}$, $L_{\mu,i}$, $L_{f,i}$ and $L_g$ according to Assumption~\ref{coefficient:assume:general}.

Rewrite the FBSDE \eqref{fbsde:decouple:origin} in the following form:
\begin{align}\label{fbsde:equ:var}
    \begin{cases}
        \ol{X}_s = \eta + \int^s_t \ol{b}(u,\ol{X}_u,\ol{Y}_u,\ol{Z}_u)\,\od u + \int^s_t \ol{\sz}(u,\ol{X}_u,\ol{Y}_u,\ol{Z}_u)\,\od \ol{W}_u;\\
        \ \\
        \ol{Y}_s = \ol{g}(\ol{X}_T) + \int^T_s \ol{f}(u,\ol{X}_u,\ol{Y}_u,\ol{Z}_u)\,\od u - \int^T_s \ol{Z}_u\,\od \ol{W}_u,
      \end{cases} 
\end{align}
with two series of coefficients $(\ol{b},\ol{\sz},\ol{f},\ol{g})$ and $(\ol{b^{\vz}},\ol{\sz^{\vz}},\ol{f^{\vz}},\ol{g^{\vz}})$ defined by
\begin{align}\label{coefficients:origin}
    \begin{cases}
        &\ol{b}(u,x,y,(z_1,z_2)) := b(u,x,y,z_1 + c(u)z_2),\\
        &\ol{\sz}(u,x,y,(z_1,z_2)) := \Sigma(u,0,x,y,(z_1,z_2)),\\
        &\ol{f}(u,x,y,(z_1,z_2)) := f(u,x,y,z_1 + c(u)z_2),\\
        &\ol{g}(x) := g(x),
    \end{cases}
\end{align}
and
\begin{align}\label{coefficients:couple}
    \begin{cases}
        &\ol{b^{\vz}}(u,x,y,(z_1,z_2)) := b^{\vz}(u,x,y,z_1 + c(u)z_2),\\
        &\ol{\sz^{\vz}}(u,x,y,(z_1,z_2)) := \Sigma^{\vz}(u,\vz(u),x,y,(z_1,z_2)),\\
        &\ol{f^{\vz}}(u,x,y,(z_1,z_2)) := f^{\vz}(u,x,y,z_1 + c(u)z_2),\\
        &\ol{g^{\vz}}(x) := g^{\vz}(x),
    \end{cases}
\end{align}
where
$$c(u) := \frac{1 - \sqrt{1 - (\vz(u))^2}}{\vz(u)}\1_{\vz(u)\neq 0},$$
$\Sigma$ is defined as in Lemma \ref{lemma:vz:sz}, and $\Sigma^{\vz} \in \cC_{\vz}(\Sigma)$ chosen according to Proposition \ref{prop:3.4}.
It is easy to see that 
\begin{itemize}
    \item 
$(X,Y, (Z,0))$ is a solution for \eqref{fbsde:equ:var}
with the coefficients $(\ol{b},\ol{\sz},\ol{f},\ol{g})$ and the initial value $\xi$;
    \item
$(X^{\vz},Y^{\vz}, (\sqrt{1 - (\vz(\cdot))^2}\,Z^{\vz},\vz(\cdot)Z^{\vz}))$
is a solution for \eqref{fbsde:equ:var}
with the coefficients $(\ol{b^{\vz}},\ol{\sz^{\vz}},\ol{f^{\vz}},\ol{g^{\vz}})$ and the initial value $\xi^{\vz}$.
\end{itemize}

We remark that $(b,\sz,A,f,g)$ and $(b^{\vz},\sz^{\vz},A^{\vz},f^{\vz},g^{\vz})$ share the same Lipschitz constants,
$\ol{g}$ and $\ol{g^{\vz}}$ admit the same Lipschitz constant $L_g$.
Moreover, notice that 
$$0\leq c(u) \leq \vz(u) \leq 1,$$ 
thus, for any $u\in [t,T]$ and 
$$(x,y,(z_1,z_2)),\ (x',y',(z'_1,z'_2)) \in \R^n\times\R^m\times (\R^{m\times d})^2,$$
we have
\begin{align}\label{lip:const:b}
    &|\ol{b^{\vz}}(u,x,y,z) - \ol{b^{\vz}}(u,x',y',z')|\notag\\
    &\quad = |b^{\vz}(u,x,y,z_1 + c(u)z_2) - b^{\vz}(u,x',y',z'_1 + c(u)z'_2)|\notag\\
    &\quad \leq L_{b,1} |x - x'| + L_{b,2} |y - y'| + L_{b,3} |z_1 - z'_1 + c(u)(z_2 - z'_2)|\notag\\
    &\quad \leq L_{b,1} |x - x'| + L_{b,2} |y - y'| + \sqrt{2}L_{b,3} |z - z'|.
\end{align}
Similarly, \eqref{lip:const:b} holds for $\ol{b}$, 
and for $\ol{f}$ and $\ol{f^{\vz}}$ with $L_{b,i}$ replaced by $L_{f,i}$, $i\in\{1,2,3\}$.

For the coefficients $\ol{\sz}$ and $\ol{\sz^{\vz}}$, by Lemma \ref{lemma:vz:sz},
\begin{align}\label{lip:const:sz}
    &|\ol{\sz^{\vz}}(u,x,y,(z_1,z_2)) - \ol{\sz^{\vz}}(u,x',y',(z'_1,z'_2))|\notag\\
    &\quad \leq L_{\mu,1}|x - x'| + L_{\mu,2}|y - y'| + L_{\mu,3}\left|(z_1,z_2) - (z'_1,z'_2)\right|,
\end{align}
and the same result holds for $\ol{\sz}$. 

\smallskip

Then, by Lemma \ref{p-exist}, there is a constant $\dz \in (0,\dz_0)$ such that, if $T - t \leq \dz$,
there is a unique $\cL^p$-solution for \eqref{fbsde:equ:re}, namely $(X,Y,Z)$.
By Lemma \ref{lemm:p-estimate},
we can further conclude that
\begin{align}\label{est:equ}
    &\E\left[|(X - X^{\vz})^*_T|^p + |(Y - Y^{\vz})^*_T|^p \right.\notag\\
    &\qquad \left.+ \left(\int^T_t |(Z_u,0) - (\sqrt{1 - (\vz(u))^2}Z^{\vz}_u,\vz(u)Z^{\vz}_u)|^2\,\od u\right)^{p/2}\right]\notag\\
    &\quad \leq c_{\eqref{p-var:equ}} \E\left[|\xi - \xi^{\vz}|^p + |g(X_T) - g^{\vz}(X_T)|^p + V_p([t,T])\right],
\end{align}
where
\begin{align*}
    V_p([t,T]) 
    &:= \left(\int^T_t |\ol{b}(u,X_u,Y_u,(Z_u,0)) - \ol{b^{\vz}}(u,X_u,Y_u,(Z_u,0))|\,\od u\right)^p \\
    &\qquad + \left(\int^T_t |\ol{\sz}(u,X_u,Y_u,(Z_u,0)) - \ol{\sz^{\vz}}(u,X_u,Y_u,(Z_u,0))|^2\,\od u\right)^{p/2}\\
    &\qquad + \left(\int^T_t |\ol{f}(u,X_u,Y_u,(Z_u,0)) - \ol{f^{\vz}}(u,X_u,Y_u,(Z_u,0))|\,\od u\right)^p\\
    &= \left(\int^T_t |b(u,X_u,Y_u,Z_u) - b^{\vz}(u,X_u,Y_u,Z_u)|\,\od u\right)^p \\
    &\qquad + \left(\int^T_t |f(u,X_u,Y_u,Z_u) - f^{\vz}(u,X_u,Y_u,Z_u)|\,\od u\right)^p\\
    &\qquad + \left(\int^T_t |\Sigma(u,0,X_u,Y_u,(Z_u,0)) - \Sigma^{\vz}(u,\vz(u),X_u,Y_u,(Z_u,0))|^2\,\od u\right)^{p/2}\\
    & =: I_b + I_f + I_{\Sigma}.
\end{align*}
For the RHS of \eqref{est:equ}, it is suffices to deal with the term $I_{\Sigma}$. Notice that
\begin{align*}
    I_{\Sigma}
    &\leq \left(\int^T_t \left|I_{\Sigma,1}(u) + I_{\Sigma,2}(u)\right|^2\,\od u\right)^{p/2}
\end{align*}
where
\begin{align*}
    I_{\Sigma,1}(u) 
    &:=  \left|\left(\sz(u,X_u,Y_u) + A(u,Z_u)\right) - \left(\sqrt{1 - (\vz(u))^2} \sz^{\vz}(u,X_u,Y_u) + A^{\vz}(u,Z_u)\right)\right|\\
    &\leq \left|\left(\sz(u,X_u,Y_u) + A(u,Z_u)\right) - \left(\sz^{\vz}(u,X_u,Y_u) + A^{\vz}(u,Z_u)\right)\right|\\
    &\qquad + (1 - \sqrt{1 - (\vz(u))^2})\left|\sz^{\vz}(u,X_u,Y_u)\right|,
\end{align*}
and
\begin{align*}
    I_{\Sigma,2}(u) 
    &:= \left|\vz(u) \sz^{\vz}(u,X_u,Y_u)\right| = \vz(u) \left|\sz^{\vz}(u,X_u,Y_u)\right|.
\end{align*}
By the fact that $1 - \sqrt{1 - (\vz(u))^2} \leq \vz(u)$
and Proposition \ref{prop:2.5}, we have
\begin{align*}
    \left(I_{\Sigma}\right)^{2/p}
    &\leq 2\int^T_t |\Dz \sz(u,X_u,Y_u) + \Dz A(u,Z_u)|^2\,\od u\\
    &\qquad + 8\int^T_t (\vz(u))^2|\sz^{\vz}(u,X_u,Y_u)|^2\,\od u\\
    &\leq 2\int^T_t |\Dz \sz(u,X_u,Y_u) + \Dz A(u,Z_u)|^2\,\od u\\
    &\qquad + 24\int^T_t (\vz(u))^2(\sz^{\vz}(u,0,0))^2\,\od u\\
    &\qquad + 24\int^T_t (\vz(u))^2\,\od u \left[\sup_{u\in [t,T]}\left((L_{\mu,1})^2|X_u|^2 + (L_{\mu,2})^2|Y_u|^2\right)\right].
\end{align*}
For the LHS of \eqref{est:equ}, the following estimations hold:
\begin{align}\label{equ:Z_process}
    & |(Z_u,0) - (\sqrt{1 - (\vz(u))^2}Z^{\vz}_u,\vz(u)Z^{\vz}_u)|^2\notag\\
    & \quad = |Z^{\vz}_u|^2 - 2\sqrt{1 - (\vz(u))^2}Z^{\vz}_u\cdot Z_u + |Z_u|^2\notag\\
    & \quad = (1 - \sqrt{1 - (\vz(u))^2})[|Z^{\vz}_u|^2 + |Z_u|^2] + \sqrt{1 - (\vz(u))^2}|Z^{\vz}_u - Z_u|^2\notag\\
    & \quad \geq \frac{1 + \sqrt{1 - (\vz(u))^2}}{2}|Z^{\vz}_u - Z_u|^2 \geq \frac{|Z^{\vz}_u - Z_u|^2}{2},
\end{align}
and
\begin{align}\label{equ:Z_process:vz}
    &|(Z_u,0) - (\sqrt{1 - (\vz(u))^2}Z^{\vz}_u,\vz(s)Z^{\vz}_u)|^2\notag\\
    & \quad \geq (1 - \sqrt{1 - (\vz(u))^2})[|Z^{\vz}_u|^2 + |Z_u|^2]\notag\\
    & \quad \geq (1 - \sqrt{1 - (\vz(u))^2})|Z_u|^2.
\end{align}
Thus, by \eqref{est:equ}, \eqref{equ:Z_process}, \eqref{equ:Z_process:vz} and Proposition \ref{prop:2.5},
we conclude that
\begin{align*}  
    CV_p([t,T])
    &\leq c_{\eqref{p-var:equ}}c'_p\left\{\E\left[|\xi - \xi^{\vz}|^p + |g(X_T) - g^{\vz}(X_T)|^p\right]\right.\\
    &\quad\qquad + \left.\E\left[U_p([t,T])\right] + S_{p,\vz}([t,T])\right\},
\end{align*}
where $c'_p > 1$ is a constant only relies on $p$,
$U_p([t,T])$ is as defined previously, and
\begin{align*}
    S_{p,\vz}([t,T])
    &:= \left(\int^T_t (\vz(u))^2\,\od u\right)^{p/2}\\
    &\qquad \times \E\left[(L_{\mu,1})^p\sup_{u\in [t,T]}|X_u|^p + (L_{\mu,2})^p\sup_{u\in [t,T]}|Y_u|^p\right]\\
    &\qquad + \E\left[\left(\int^T_t (\vz(u))^2(\sz^{\vz}(u,0,0))^2\,\od u\right)^{p/2}\right].
\end{align*}
By Proposition~\ref{prop:3.4}, we can replace $\sz^{\vz}$ by $\sz$ for the above equality,
and this finishes the proof for the estimation.
\end{proof}

\subsection{Transference of Decoupling Field}\label{sec:decoupling_vs_coupling}

Now we investigate the coupling variance for
arbitrary length time intervals.
And we use the decoupling field as a background for this topic.

Before we present the main theorem in Section \ref{sec:arbitrary},
in this section, we analyze the transference of the decoupling field structure
under the coupling operator $\cC_{\vz}$ as defined in Section~\ref{sec:specialSetting}.

For clarity, we introduce the following notations:
\begin{itemize}
    \item $\cC_{\varphi,0}$ denotes the coupling operator applied to random variables;
    \item $\cC_{\varphi,T}$ denotes the coupling operator applied to stochastic processes.
\end{itemize}
We will use these distinctions to systematically study how the coupling method interacts with the structural properties of decoupling fields.
Notice that the coupling operator only maps the predictable process to predictable process (Proposition \ref{prop:2.12}),
the decoupling field $u$ may only be progressive measurable, then the measurability of $\cC_{\vz,T}(u)$ is unknown.
Thus, we need to find a suitable substitute. 
The main lemma is as follows:

\begin{lemm}\label{lemma:technic:de}
    Suppose the FBSDE \eqref{fbsde:equ:re} satisfies all the following conditions: 
    \begin{itemize}
        \item[(a)] Assumption~\ref{coefficient:assume:general}, Assumption~\ref{fbsde:basic:assume:p:general} and Assumption~\ref{ass:linear:diffusion} hold with $p = 2$.
        \item[(b)] The equation admits a regular decoupling field $w$ with $L_w L_{\mu,3} < 1$.
    \end{itemize}
    Then, there exists a constant $\ol{\dz}>0$ such that,
    for any partition $t = t_0 \leq t_1\leq\cdots\leq t_n = T$ satisfying $t_{i + 1} - t_i \leq \ol{\dz}$, 
    there exists a sequence of measurable functions $\{w^{\vz,i}\}^n_{i = 0}$ satisfying the following properties:
    For each $i\in\{0,1,\dots,n\}$, the following hold
    \begin{itemize}
        \item[(1)] Coupling Consistency:
        $w^{\vz,i}(\cdot,x)\in \cC_{\vz,0}(w(t_i,\cdot,x))$ for any $x\in\R^n$;
        \item[(2)] Lipschitz Continuity:
        For all $x,y\in\R^n$ and $\oz\in\Oz$,
        \begin{equation}\label{equ:Lipschitz:ui}
            |w^{\vz,i}(\oz,x) - w^{\vz,i}(\oz,x)| \leq L_w |x - y|;
        \end{equation}
        \item[(3)] Existence and Uniqueness of Local FBSDE Solution:
        For any $\ol{\xi}\in \cL^2(\cF^{\ol{W}}_{t_i})$,
        where $\ol{W} := 
        \begin{pmatrix}
        W \\
        W'
        \end{pmatrix}$,
        the following FBSDE
    \begin{align}\label{fbsde:equ:de:field}
    \begin{cases}
        \ol{X}_s = \ol{\xi} + \int^s_{t_i} \ol{b^{\vz}}(u,\ol{X}_u,\ol{Y}_u,\ol{Z}_u)\,\od u + \int^s_{t_i} \ol{\sz^{\vz}}(u,\ol{X}_u,\ol{Y}_u,\ol{Z}_u)\,\od \ol{W}_u;\\
        \ \\
        \ol{Y}_s = w^{\vz,i}(\ol{X}_{t_{i + 1}}) + \int^{t_{i + 1}}_s \ol{f^{\vz}}(u,\ol{X}_u,\ol{Y}_u,\ol{Z}_u)\,\od u - \int^{t_{i + 1}}_s \ol{Z}_u\,\od \ol{W}_u
    \end{cases}
    \end{align}
    admits a unique solution $(\ol{X}^{\ol{\xi},[t_i,t_{i + 1}]},\ol{Y}^{\ol{\xi},[t_i,t_{i + 1}]},\ol{Z}^{\ol{\xi},[t_i,t_{i + 1}]})$
    where the coefficients $(\ol{b^{\vz}},\ol{\sz^{\vz}},\ol{f^{\vz}})$ is defined as in \eqref{coefficients:couple};
    \item[(4)] Decoupling Field Property:
    The solution in $(3)$ satisfies,
    for any $\ol{\xi}\in \cL^2(\cF^{\ol{W}}_{t_i})$,
    \begin{equation}
        \ol{Y}^{\ol{\xi},[t_i,t_{i + 1}]}_{t_i} = w^{\vz,i}(\ol{\xi})\quad \P-a.s.;
    \end{equation}
    \item[(5)] Structure of $\ol{Z}$:
    The solution in $(3)$ satisfies,
    there is a $\R^{m\times d}$ predictable process $\wt{Z}\in \cL^{2,2}$ such that 
    $$\ol{Z} = (\sqrt{1 - (\vz(\cdot))^2}\,\wt{Z}, \vz(\cdot)\wt{Z}).$$
    \end{itemize}
\end{lemm}
\begin{proof}
\textbf{Construction of $w^{\vz,i}$ and proof of (2)}: 
This follows from the Proposition~\ref{prop:2.5}(4) (the existence of function $u^{\vz,i}$),
and Proposition~\ref{prop:3.4} (the Lipschitz property).

\smallskip

\textbf{Proof of Lemma \ref{lemma:technic:de}(3), (4) and (5)}: 
    
By \eqref{equ:Lipschitz:ui} and \eqref{lip:const:sz},
$(w^{\vz,i},\ol{\sz^{\vz}})$ share the same Lipschitz constants as $(w,\sz)$.
By \eqref{lip:const:b}, for any $h\in\{b,f\}$,
    \begin{align}\label{lip:const:b:0}
        &|\ol{h^{\vz}}(s,x,y,z) - \ol{h^{\vz}}(s,x',y',z')|\notag\\
        &\quad \leq L_{h,1} |x - x'| + L_{h,2} |y - y'| + \sqrt{2}L_{h,3} |z - z'|.
    \end{align}

    By \eqref{equ:Lipschitz:ui} and the fact that $L_wL_{\mu,3} < 1$,
    Lemma \ref{p-exist} shows that,
    there is $\wt{\dz} > 0$, 
    such that \eqref{fbsde:equ:de:field} admits a unique solution if $t_{i + 1} - t_i\leq \wt{\dz}$.
    
    Let $\dz$ be the constant defined as in Definition \ref{def:dfield}.
    Set
    $$\ol{\dz} := \min\{\dz,\wt{\dz}\}.$$
    Then, if $t_{i + 1} - t_i \leq \ol{\dz}$, for any $x_0\in\R^n$,
    the FBSDE
    \begin{align}\label{fbsde:equ:decouple:field:small}
        \begin{cases}
            X_s = x_0 + \int^{s}_{t_i} b(u,X_u,Y_u,Z_u)\,\od u + \int^{s}_{t_i} \sz(u,X_u,Y_u) + A(u,Z_u)\,\od W_u;\\
            \ \\
            Y_s = w(t_{i + 1},X_{t_{i + 1}}) + \int^{t_{i + 1}}_s f(u,X_u,Y_u,Z_u)\,\od u - \int^{t_{i + 1}}_s Z_u\,\od W_u,
          \end{cases} 
    \end{align}
    admits a unique solution 
    $$(X^{x_0,[t_i,t_{i + 1}]},Y^{x_0,[t_i,t_{i + 1}]},Z^{x_0,[t_i,t_{i + 1}]}).$$
    And, similar as in the proof of Theorem \ref{thm:small:decoupling},
    we conclude that
    $$\left(X^{x_0,\vz,[t_i,t_{i + 1}]},Y^{x_0,\vz,[t_i,t_{i + 1}]},(\sqrt{1 - (\vz(\cdot))^2}\,Z^{x_0,\vz,[t_i,t_{i + 1}]},\vz(\cdot)\,Z^{x_0,\vz,[t_i,t_{i + 1}]})\right)$$
    is the solution of \eqref{fbsde:equ:de:field} with the initial value $\ol{\xi} = x_0$, where
    $$H^{x_0,\vz,[t_i,t_{i + 1}]} \in \cC_{\vz,T}\left(H^{x_0,[t_i,t_{i + 1}]}\right),\quad H\in\{X,Y,Z\}.$$
    Moreover, the definition of $\ol{\dz}$ shows that the solution above is also unique.

    By Lemma \ref{lemma:technic:de}$(1)$ and 
    the property of decoupling field, i.e. \eqref{dfield:property}, we have 
    \begin{align*}
        &w^{\vz,i}(x_0) \in \cC_{\vz,0}(w(t_i,x_0)),\\
        &Y^{x_0,[t_i,t_{i + 1}]}_{t_i} = w(t_i,x_0)\quad \P-a.s.
    \end{align*}
    Thus, by Proposition \ref{prop:2.5},
    \begin{equation}\label{equ:de:dfield}
        w^{\vz,i}(x_0) = Y^{x_0,\vz,[t_1,t_2]}_{t_i} = \ol{Y}^{x_0,[t_i,t_{i + 1}]}_{t_i}\quad \ol{\P}-a.s.
    \end{equation}
    
    For any $\ol{\xi}\in \cL^2(\cF^{\ol{W}}_{t_i})$, 
    the existence and the uniqueness of the solution to \eqref{fbsde:equ:de:field}
    follow from the definition of $\ol{\dz}$. This completes the proof of Lemma \ref{lemma:technic:de}(3).

    \smallskip

    By \eqref{equ:de:dfield}, Lemma \ref{lemma:technic:de}(4) holds when $\ol{\xi} \in\R^n$.
    Notice that $w^{\vz,i}$ is Lipschitz, for the general $\ol{\xi}$,
    using the standard approximation method with simple function
    (see \cite[Step 3 in the Proof of Theorem 11.3.4]{Cvitanic:Zhang:13} for example).
    This completes the proof of Lemma \ref{lemma:technic:de}(4). 

    \smallskip

    It is clear that Lemma \ref{lemma:technic:de}(5) holds when $\ol{\xi} \in\R^n$.
    Thus, for any simple function $\ol{\xi} = \sum^l_{k = 1}x_k\1_{A_k}\in \cL^2(\cF^{\ol{W}}_{t_i})$, Lemma \ref{lemma:technic:de}(5) still holds.
    For the general $\ol{\xi}\in \cL^2(\cF^{\ol{W}}_{t_i})$, let 
    $$(\ol{X}^{\ol{\xi},[t_i,t_{i + 1}]},\ol{Y}^{\ol{\xi},[t_i,t_{i + 1}]},\ol{Z}^{\ol{\xi},[t_i,t_{i + 1}]})$$
    be the unique solution of FBSDE \eqref{fbsde:equ:de:field}.
    Consider a sequence of simple functions $\{\ol{\xi}_l\}\subset \cL^2(\cF^{\ol{W}}_{t_i})$ with $\ol{\xi}_l\to\ol{\xi}$.
    Then, by the fact that Lemma \ref{lemma:technic:de}(5)  holds for simple function,
    for any $l\in\bN$, there is a predictable process on $[t,T]$ denoted by $\wt{Z}^{\ol{\xi}_{l},\vz,[t_1,t_2]}\in\cL^{2,2}$ such that
    $$\ol{Z}^{\ol{\xi}_{l},\vz,[t_1,t_2]} = \left(\sqrt{1 - (\vz(\cdot))^2}\, \wt{Z}^{\ol{\xi}_{l},\vz,[t_1,t_2]}, \vz(\cdot)\wt{Z}^{\ol{\xi}_{l},\vz,[t_1,t_2]}\right)$$  
    Moreover, by Lemma \ref{lemm:p-estimate}, for any $q\in\bN$,
    \begin{align*}
        &\E\left[\int^{t_2}_{t_1} \left|\wt{Z}^{\ol{\xi}_{l + q},\vz,[t_1,t_2]}_s - \wt{Z}^{\ol{\xi}_l,\vz,[t_1,t_2]}_s\right|^2\,\od s\right]\\
        &\quad = \E\left[\int^{t_2}_{t_1} \left|\ol{Z}^{\ol{\xi}_{l + q},\vz,[t_1,t_2]}_s - \ol{Z}^{\ol{\xi}_l,\vz,[t_1,t_2]}_s\right|^2\,\od s\right]
        \lesssim \E\left[|\ol{\xi}_{l + q} - \ol{\xi}_l|^2\right].
    \end{align*}
    Thus, there is a predictable process $\wt{Z}$ such that $\wt{Z}^{\ol{\xi}_l,\vz,[t_1,t_2]} \to \wt{Z}$ in $\cL^{2,2}$.
    Notice that 
    $$\ol{Z}^{\ol{\xi}_l,\vz,[t_1,t_2]} \to \ol{Z}^{\ol{\xi},\vz,[t_1,t_2]}\quad\text{in}\quad \cL^{2,2}$$
    and 
    $$\ol{Z}^{\ol{\xi}_l,\vz,[t_1,t_2]} \to  (\sqrt{1 - (\vz(\cdot))^2}\,\wt{Z}, \vz(\cdot)\wt{Z})\quad\text{in}\quad \cL^{2,2}.$$
    Then $\ol{Z}^{\ol{\xi},\vz,[t_1,t_2]} = (\sqrt{1 - (\vz(\cdot))^2}\,\wt{Z}, \vz(\cdot)\wt{Z})$ in $\cL^{2,2}$.
    And this completes the proof of Lemma \ref{lemma:technic:de}(5). 
\end{proof}

\subsection{On arbitrary finite time interval}
\label{sec:arbitrary}

In this section, with the help of the decoupling field,
we prove the main theorem for the time intervals with arbitrary finite length.
We estimate the estimation for $p$-coupling variance $CV_p([t,T])$ of the FBSDE \eqref{fbsde:equ:re}
for any $T\in (0,\fz)$.
\begin{theo}\label{thm:large:decoupling}
    If the FBSDE \eqref{fbsde:equ:re}
    \begin{itemize}
        \item[(1)] satisfies Assumption \ref{coefficient:assume:general}, Assumption \ref{fbsde:basic:assume:p:general} with some $p\in [2,\fz)$ and Assumption \ref{ass:linear:diffusion};
        \item[(2)] admits a regular decoupling field $w$ with the Lipschitz constants $L_w$;
        \item[(3)] \eqref{p-condition} holds with $L_g$ replaced by $L_w$.
    \end{itemize}
    Let $\Tz := (X,Y,Z)$ be the solution of \eqref{fbsde:equ:re}.
    Then the estimation for $p$-coupling variance holds
    \begin{align}\label{equ:large:decoupling}
        CV_p([t,T])
        &\leq c_{\eqref{equ:large:decoupling}} \left\{\E\left[|\xi - \xi^{\vz}|^p + |g(X_T) - g^{\vz}(X_T)|^p\right] \right.\notag\\
        &\quad \qquad \left. + \E\left[U_p([t,T])\right] + S_p([t,T])\right\},
    \end{align}
    where the constant $c_{\eqref{equ:large:decoupling}}$ only relies on $p$, $T - t$
    and all the Lipschitz constant in Assumption \ref{coefficient:assume:general},
    and $U_p([t,T])$ and $S_p([t,T])$ are defined as in Theorem \ref{thm:small:decoupling}.
\end{theo}
\begin{proof}
    Let $\dz := \min\{\dz_1,\dz_2\}$ 
    where 
    \begin{itemize}
        \item[(1)] $\dz_1$ is the constant given in Lemma \ref{lemma:technic:de};
        \item[(2)] $\dz_2$ is the constant given in Theorem \ref{thm:small:decoupling}.
    \end{itemize}
    Consider a partition $t = t_0 < t_1 < t_2 < \cdots < t_n = T$ satisfying $t_{i + 1} - t_i \leq \dz$ for any $i\in\{0,\cdots,n - 1\}$,
    the solution of FBSDE \eqref{fbsde:equ:re} on $[t_i,t_{i + 1}]$ satisfies
    \begin{align}\label{fbsde:equ:small}
        \begin{cases}
            X_s = X_{t_i} + \int^s_{t_i} b(u,X_u,Y_u,Z_u)\,\od u + \int^s_{t_i} \sz(u,X_u,Y_u) + A(u,Z_u)\,\od W_u;\\
            \ \\
            Y_s = w(t_{i + 1},X_{t_{i + 1}}) + \int^{t_{i + 1}}_s f(u,X_u,Y_u,Z_u)\,\od u - \int^{t_{i + 1}}_s Z_u\,\od W_u,
          \end{cases} 
    \end{align}
    By the fact that the decoupling field $w$ is regular with $L_w L_{\mu,3} < 1$,
    Theorem \ref{thm:small:decoupling} shows that, for any $i\in\{0,1,\cdots,n - 1\}$,
    \begin{align*}
        CV_p([t_i,t_{i + 1}]) 
        &\leq c_{\eqref{equ:small:decoupling}}\E\left[|X_{t_i} - X^{\vz}_{t_i}|^p + |w(t_{i + 1},X_{t_{i + 1}}) - w^{\vz}(t_{i + 1},X_{t_{i + 1}})|^p\right]\\
        &\qquad  + c_{\eqref{equ:small:decoupling}}\E\left[U_p([t_i,t_{i + 1}])\right] + c_{\eqref{equ:small:decoupling}}S_p([t_i,t_{i + 1}]).
    \end{align*}
    W.l.o.g, we assume that $T - t = N\dz$ for some $N\in\bN^+$ and $t_i = t + i\dz$.
    Then for any $i\in\{0,\cdots,N - 1\}$, using Theorem \ref{thm:small:decoupling} recursively,
    notice that the constant $c_{\eqref{equ:small:decoupling}} \geq 1$, we have
    \begin{align*}
        &CV_p([t_i,t_{i + 1}])\\
        &\quad\leq c_{\eqref{equ:small:decoupling}}\E\left[|X_{t_i} - X^{\vz}_{t_i}|^p\right]\\
        &\quad\qquad + c_{\eqref{equ:small:decoupling}}\E\left[|w(t_{i + 1},X_{t_{i + 1}}) - w^{\vz}(t_{i + 1},X_{t_{i + 1}})|^p\right]\\
        &\quad\qquad + c_{\eqref{equ:small:decoupling}}\E\left[U_p([t_i,t_{i + 1}])\right] + c_{\eqref{equ:small:decoupling}}S_p([t_i,t_{i + 1}])\\
        &\quad\leq c_{\eqref{equ:small:decoupling}} CV_p([t_{i - 1},t_i])\\
        &\quad\qquad + c_{\eqref{equ:small:decoupling}}\E\left[|w(t_{i + 1},X_{t_{i + 1}}) - w^{\vz}(t_{i + 1},X_{t_{i + 1}})|^p\right]\\
        &\quad\qquad + c_{\eqref{equ:small:decoupling}}\E\left[U_p([t_i,t_{i + 1}])\right] + c_{\eqref{equ:small:decoupling}}S_p([t_i,t_{i + 1}])\\
        &\quad\leq (c_{\eqref{equ:small:decoupling}})^2 CV_p([t_{i - 2},t_{i - 1}])\\
        &\quad\qquad + (c_{\eqref{equ:small:decoupling}})^2 \sum^i_{k = i - 1}\E\left[|w(t_{k + 1},X_{t_{k + 1}}) - w^{\vz}(t_{k + 1},X_{t_{k + 1}})|^p\right]\\
        &\quad\qquad + (c_{\eqref{equ:small:decoupling}})^2 \sum^i_{k = i - 1}\E\left[U_p([t_k,t_{k + 1}])\right] + (c_{\eqref{equ:small:decoupling}})^2 \sum^i_{k = i - 1}S_p([t_k,t_{k + 1}])\\
        &\quad\leq\cdots\\
        &\quad\leq (c_{\eqref{equ:small:decoupling}})^{i + 1} \E\left[|\xi - \xi^{\vz}|^p\right]\\
        &\quad\qquad + (c_{\eqref{equ:small:decoupling}})^{i + 1} \sum^i_{k = 0}\E\left[|w(t_{k + 1},X_{t_{k + 1}}) - w^{\vz}(t_{k + 1},X_{t_{k + 1}})|^p\right]\\
        &\quad\qquad + (c_{\eqref{equ:small:decoupling}})^{i + 1} \sum^i_{k = 0}\E\left[U_p([t_k,t_{k + 1}])\right]+ (c_{\eqref{equ:small:decoupling}})^{i + 1} \sum^i_{k = 0}S_p([t_k,t_{k + 1}])\\
        &\quad\leq (c_{\eqref{equ:small:decoupling}})^{N + 1} \E\left[|\xi - \xi^{\vz}|^p\right]\\
        &\quad\qquad + (c_{\eqref{equ:small:decoupling}})^{N + 1} \sum^{N - 1}_{k = 0}\E\left[|w(t_{k + 1},X_{t_{k + 1}}) - w^{\vz}(t_{k + 1},X_{t_{k + 1}})|^p\right]\\
        &\quad\qquad + (c_{\eqref{equ:small:decoupling}})^{N + 1} \E\left[U_p([t,T])\right] + N(c_{\eqref{equ:small:decoupling}})^{N + 1} S_p([t,T]),
    \end{align*}
    where the last inequality holds by Aoki--Rolewicz inequality and the fact that $p\in [2,\fz)$.
    Thus, we have
    \begin{align}\label{equ:est:01}
        \frac{CV_p([t,T])}{N^{p/2 - 1}}
        & \leq \sum^{N - 1}_{k = 0}CV_p([t_k,t_{k + 1}])\notag\\
        & \leq N(c_{\eqref{equ:small:decoupling}})^{N + 1}\E\left[|\xi - \xi^{\vz}|^p\right]\notag\\
        &\qquad + N(c_{\eqref{equ:small:decoupling}})^{N + 1} \sum^{N - 1}_{k = 0}\E\left[|w(t_{k + 1},X_{t_{k + 1}}) - w^{\vz}(t_{k + 1},X_{t_{k + 1}})|^p\right]\notag\\
        &\qquad + N(c_{\eqref{equ:small:decoupling}})^{N + 1} \E\left[U_p([t,T])\right] + N^2(c_{\eqref{equ:small:decoupling}})^{N + 1} S_p([t,T]).
    \end{align}

    Now, it is suffices to estimate the term 
    $$\E\left[|w(t_{i + 1},X_{t_{i + 1}}) - w^{\vz}(t_{i + 1},X_{t_{i + 1}})|^p\right].$$
    By Lemma \ref{lemma:technic:de}, for any $i\in\{1,2,\dots,N\}$, we have
    $$w^{\vz}(t_i,X_{t_i}) = w^{\vz,i}(X_{t_i}) = \overline{Y}_{t_i}\quad\P-a.s.,$$
    where $w^{\vz,i}$ is defined as in Lemma \ref{lemma:technic:de} and
    $\overline{\Tz} := (\overline{X},\overline{Y},\overline{Z})$ is the solution of the FBSDE \eqref{fbsde:equ:de:field} with the initial value $\ol{\xi} = X_{t_i}$.

    Consider the equation \eqref{fbsde:equ:var} with $g$ replaced by $w^{\vz}(t_{i + 1},\cdot)$,
    following the same method in the proof of Theorem \ref{thm:small:decoupling}, we have
    \begin{align*}
        &\E\left[|w(t_i,X_{t_i}) - w^{\vz}(t_i,X_{t_i})|^p\right]\\
        &\quad = \E\left[\left|Y_{t_i} - \overline{Y}_{t_i}\right|^p\right]\\
        &\quad \leq c_{\eqref{equ:small:decoupling}}\E\left[|w(t_{i + 1},X_{t_{i + 1}}) - w^{\vz}(t_{i + 1},X_{t_{i + 1}})|^p\right]\\
        &\quad\qquad + c_{\eqref{equ:small:decoupling}} \E\left[U_p[t_i,t_{i + 1}]\right] + c_{\eqref{equ:small:decoupling}} S_p([t_i,t_{i + 1}])\\
        &\quad \leq \cdots\\
        &\quad \leq (c_{\eqref{equ:small:decoupling}})^{N - i}\E\left[|g(X_T) - g^{\vz}(X_T)|^p\right]\\
        &\quad\qquad + (c_{\eqref{equ:small:decoupling}})^{N - i}\E\left[U_p([t_i,T])\right] + (N - i)(c_{\eqref{equ:small:decoupling}})^{N - i} S_p([t_i,T])\\
        &\quad \leq (c_{\eqref{equ:small:decoupling}})^N\E\left[|g(X_T) - g^{\vz}(X_T)|^p\right]\\
        &\quad\qquad + (c_{\eqref{equ:small:decoupling}})^N\E\left[U_p([t,T])\right] + N(c_{\eqref{equ:small:decoupling}})^N S_p([t,T]),
    \end{align*}
    then
    \begin{align}\label{equ:est:02}
        &\sum^{N - 1}_{i = 0}\E\left[|w(t_{i + 1},X_{t_{i + 1}}) - w^{\vz}(t_{i + 1},X_{t_{i + 1}})|^p\right]\notag\\
        &\quad \leq N(c_{\eqref{equ:small:decoupling}})^N\E\left[|g(X_T) - g^{\vz}(X_T)|^p\right]\notag\\
        &\quad\qquad + N(c_{\eqref{equ:small:decoupling}})^N \E\left[U_p([t,T])\right] + N^2(c_{\eqref{equ:small:decoupling}})^N S_p([t,T]).
   \end{align}
   By \eqref{equ:est:01} and \eqref{equ:est:02}, the desired estimation holds:
   \begin{align*}
    CV_p([t,T])
    &\leq N(c_{\eqref{equ:small:decoupling}})^{N + 1}\E\left[|\xi - \xi^{\vz}|^p\right]\\
    &\qquad + N^2(c_{\eqref{equ:small:decoupling}})^{2N + 1}\E\left[|g(X_T) - g^{\vz}(X_T)|^p\right]\\
    &\qquad + 2N^2(c_{\eqref{equ:small:decoupling}})^{2N + 1}\E\left[U_p([t,T])\right] + 2N^2(c_{\eqref{equ:small:decoupling}})^{2N + 1}S_p([t,T]).
   \end{align*}
   This finishes the proof of Theorem \ref{thm:large:decoupling}.
\end{proof}
\begin{rema}
    In some special case, Theorem \ref{thm:large:decoupling} agrees with the known results.
    \begin{itemize}
        \item[(1)] For the BSDEs case, 
        i.e. the coefficients $b$, $\mu$ are zero functions, 
        and $f$, $g$ do not rely on $x$,
        Theorem \ref{thm:large:decoupling} obtain the same estimation
        for the $p$-coupling variance as \cite[Theorem 6.3]{Geiss:Ylinen:21};
        \item[(2)] For SDEs case,
        i.e. the coefficients $f$ and $g$ are zero functions,
        and $b$, $\mu$ do not rely on $y$ or $z$,
        Theorem \ref{thm:large:decoupling} obtain the same estimation
        for the $p$-coupling variance for SDE as \cite[Theorem 5.4]{Geiss:Zhou:24}.
    \end{itemize}
 \end{rema}

 \section{Applications on Regularity}
 \label{sec:regularity}
 In this section, we turn to the application for the $p$-coupling variance.
 The main results include two kinds of regularity for the solution to FBSDEs:
 \begin{itemize}
    \item[(1)] $\cL^p$ regularity in time;
    \item[(2)] $\D_{1,2}$ differentiability.
 \end{itemize}
 
In this section, in order to shorten the notation, we again use
$\mu$ to represent the coefficients for the diffusion in the forward equation, i.e. 
$$\mu = \sz + A.$$
We always consider the FBSDE \eqref{fbsde:equ} with the Assumption \ref{coefficient:assume:general},
Assumption \ref{fbsde:basic:assume:p:general} and Assumption \ref{ass:linear:diffusion}.
Recall that the term $CV_p([t,T])$ denotes the $p$-coupling variance which is defined by \eqref{equ:decouplingVariance}.

In order to present the results more clearly,
we use the following abbreviations:
For all $t\leq a < b \leq T$,
\begin{align*}
    I_{p,b}([a,b]) &:= \E\left[\Biggl(\int_a^b |b^0(s)|\,\od s\Biggr)^p\right],\quad 
    I_{p,f}([a,b]) := \E\left[\Biggl(\int_a^b |f^0(s)|\,\od s\Biggr)^p\right]\\
    I_{p,\mu}([a,b]) &:= \E\left[\Biggl(\int_a^b |\mu^0(s)|^2\,\od s\Biggr)^{p/2}\right],\\
    I_{p,T} &:= \E\left[|\xi|^p + |g(0)|^p\right] + I_{p,b}([t,T]) + I_{p,f}{[t,T]} + I_{p,\mu}([t,T]),
\end{align*}
where $b^0$, $f^0$ and $\mu^0$ are defined by \eqref{equ:def:0function}.
Moreover, the full potential is given as
\begin{align*}
    P_p([t,T]) &:= \E\left[|g(X_T) - g^{\vz}(X_T)|^p + U_p([t,T])\right],
\end{align*}
with $U_p$ defined as in Theorem \ref{thm:small:decoupling}.

 \subsection{Regularity in Time}
 In this section, we aim to get the $\cL^p$ regularity in time for the solution of the FBSDE \eqref{fbsde:equ},
 i.e. the estimation of
    $$\E\left[|X_s - X_r|^p + |Y_s - Y_r|^p + \left(\int^r_s |Z_u|^2\,\od u\right)^{p/2} \right],$$
for all $t\leq s < r\leq T$, where $(X,Y,Z)$ is the solution of the FBSDE \eqref{fbsde:equ}.

By Jensen's inequality, we have
\begin{equation}\label{equ:de:path-regularity}
    \|Y_s - Y_r\|_p
    \leq \left\|Y_s - \E\left[Y_r| \cG^r_s\right]\right\|_p
         + \left\|\E\left[Y_r| \cG^r_s\right] - Y_r\right\|_p \leq 3\|Y_r - Y_s\|_p,
\end{equation} 
where $\cG^r_s$ is the $\sz$-algebra defined as in Theorem \ref{theo:path-regularity:BSDE}.
The term $\left\|\E\left[Y_r| \cG^r_s\right] - Y_r\right\|_p$ can be controlled by $\|Y_r - Y^{(s,r]}_r\|$ by Theorem \ref{theo:path-regularity:BSDE}.
For the other term $\left\|Y_s - \E\left[Y_r| \cG^r_s\right]\right\|_p$, it can be estimated by a standard method which only relies on the a prior estimation for the solution.
Though, we also find that the coupling variance can be used to improve the estimates for $\left\|Y_s - \E\left[Y_r| \cG^r_s\right]\right\|_p$.

\smallskip

In this part, we recall the notation defined as in Section~\ref{sec:specialSetting}
where $X^{(r,s]}$ denotes the element belongs in $\cC_{\vz}(X)$ with $\vz = \1_{(r,s]}$.
The main theorem is as follows.
\begin{theo}\label{theorem:path-regularity}
    Assume the FBSDE \eqref{fbsde:equ} satisfies one of the conditions:
    \begin{itemize}
        \item[(1)] all the conditions in Theorem \ref{thm:small:decoupling} with some $p\in [2,\fz)$ and sufficiently small $T - t$;
        \item[(2)] all the conditions in Theorem \ref{thm:large:decoupling} with some $p\in [2,\fz)$ and arbitrary $T\in (0,\fz)$.
    \end{itemize}
    Let $(X,Y,Z)$ be the solution of the FBSDE \eqref{fbsde:equ}.
    Then, for all $t\leq s < r\leq T$ and the coupling function $\vz = \1_{(s,r]}$, we have
    \begin{align}\label{equ:est:cv:p}
        CV_p([t,T])
        &\leq c_{\eqref{equ:est:cv:p}} \left\{P_p([t,T]) + I_{p,\mu}([s,r]) \right.\notag\\
        &\quad\quad \left. + [(L_{\mu,1})^p + (L_{\mu,2})^p] I_{p,T} (r - s)^{p/2}\right\}.
    \end{align}
    Moreover,
    \begin{align}\label{equ:path:z}
        \E\left[ \left(\int^r_s |Z_u|^2\,\od u\right)^{p/2}\right] \leq CV_p([t,T]),
    \end{align}
    \begin{align}\label{equ:path:x}
        \E\left[|X_s - X_r|^p\right]
        &\leq c_{\eqref{equ:path:x}}\left\{ 
            I_{p,b}([s,r]) + I_{p,\mu}([s,r]) + (L_{\mu,3})^pCV_p([t,T])\right.\notag\\
            &\qquad\qquad + c_{\eqref{p-estimate}} [(L_{b,1})^p + (L_{b,2})^p] I_{p,T} (r - s)^p\notag\\
            &\qquad\qquad + c_{\eqref{p-estimate}}\left[(L_{\mu,1})^p + (L_{\mu,2})^p\right]I_{p,T} (r - s)^{p/2\notag}\\
            &\left. \qquad\qquad + (L_{b,3})^{p}(r - s)^{p/2}CV_p([t,T])\right\},
    \end{align}
    and
    \begin{align}\label{equ:path:y}
        \E\left[|Y_s - Y_r|^p\right]
        & \leq c_{\eqref{equ:path:y}}\left\{
            I_{p,f}([s,r]) + CV_p([t,T])\right.\notag\\
            &\qquad\qquad + c_{\eqref{p-estimate}}\left[(L_{f,1})^p + (L_{f,2})^p\right]I_{p,T} (r - s)^p\notag\\
            &\left. \qquad\qquad + (L_{f,3})^p (r - s)^{p/2} CV_p([t,T])\right\},
    \end{align}
    where 
    \begin{itemize}
        \item the constants $L_{b,i}$, $L_{\mu,i}$ and $L_{f,i}$ with $i\in\{1,2,3\}$ are the Lipschitz constants defined as in Assumption \ref{coefficient:assume:general};
        \item the positive constant $c_{\eqref{equ:est:cv:p}}$ only rely on $p$, $c_{\eqref{equ:small:decoupling}}$ (or $c_{\eqref{equ:large:decoupling}}$) and $c_{\eqref{p-estimate}}$;
        \item the positive constants $c_{\eqref{equ:path:x}}$ and $c_{\eqref{equ:path:y}}$ only rely on $p$.
    \end{itemize}
\end{theo}
\begin{proof}
    The proofs of these two cases are similar.
    The only difference is, the case for short-length interval relies on Theorem~\ref{thm:small:decoupling},
    while the case for arbitrary finite-length time intervals relies on Theorem~\ref{thm:large:decoupling}.
    Here we only offer the detailed proof
    when the length of the time interval is sufficiently small.

    For the coupling function $\vz = \1_{(s,r]}$, 
    by Lemma~\ref{lemm:p-estimate}, we conclude that
    \begin{align*}
        S_p([t,T])
        &:= \left(\int^T_t (\vz(u))^2\,\od u\right)^{p/2}\\
        &\quad\quad \times \E\left[(L_{\mu,1})^p\sup_{u\in [t,T]}|X_u|^p + (L_{\mu,2})^p\sup_{u\in [t,T]}|Y_u|^p\right]\\
        &\quad\quad + \E\left[\left(\int^T_t (\vz(u))^2(\mu^0(u))^2\,\od u\right)^{p/2}\right]\\
        & = (r - s)^{p/2}\E\left[(L_{\mu,1})^p\sup_{u\in [t,T]}|X_u|^p + (L_{\mu,2})^p\sup_{u\in [t,T]}|Y_u|^p\right]\\
        &\quad\quad + \E\left[\left(\int^r_s (\mu^0(u))^2\,\od u\right)^{p/2}\right]\\
        & \leq c_{\eqref{p-estimate}} [(L_{\mu,1})^p + (L_{\mu,2})^p] I_{p,T} (r - s)^{p/2} + I_{p,\mu}([s,r]). 
    \end{align*}   
    Then, combined with Theorem~\ref{thm:small:decoupling}, we obtain
    \begin{align}\label{equ:est:regu:01}
        CV_p([t,T])
        & \leq c_{\eqref{equ:small:decoupling}} \left\{\E\left[|\xi - \xi^{(s,r]}|^p + P_p([t,T])\right] + S_p([t,T])\right\}\notag\\
        & \leq c_{\eqref{equ:small:decoupling}}P_p([t,T]) + c_{\eqref{equ:small:decoupling}}I_{p,\mu}([s,r])\notag\\
        &\quad\quad + c_{\eqref{equ:small:decoupling}}c_{\eqref{p-estimate}} [(L_{\mu,1})^p + (L_{\mu,2})^p] I_{p,T} (r - s)^{p/2},
    \end{align}
    where the last inequality holds due to the fact that $\xi = \xi^{(s,r]}$, and this proves \eqref{equ:est:cv:p}.

\smallskip

We first consider the regularity of the $Z$ process.
By \eqref{equ:decouplingVariance}, it holds
\begin{align*}
    \E\left[ \left(\int^r_s |Z_u|^2\,\od u\right)^{p/2}\right]
    & = \E\left[ \left(\int^T_t \left(1 - \sqrt{1 - \1_{(s,r]}(u)}\right)|Z_u|^2\,\od u\right)^{p/2}\right]\leq CV_p([t,T]).
\end{align*}

\smallskip

Next, for the $X$ process, we observe that
\begin{align*}
    \E\left[|X_s - X_r|^p\right]
    & \leq 2^{p - 1}\E\left[ \left|\int^r_s b(u,X_u,Y_u,Z_u)\,\od u\right|^p + \left|\int^r_s \mu(u,X_u,Y_u,Z_u)\,\od W_u\right|^p\right]\\
    & =: 2^{p - 1}\left[I_1(p) + I_2(p)\right].
\end{align*}
For $I_1(p)$, by H\"older's inequality and Assumption \ref{coefficient:assume:general}, we conclude that
\begin{align*}
    I_1(p) 
    &\leq \E\left[\left(\int^r_s |b^0(u)| + L_{b,1}|X_u| + L_{b,2}|Y_u| + L_{b,3}|Z_u|\,\od u\right)^p\right]\\
    &\lesssim \E\left[\left(\int^r_s |b^0(u)|\,\od u\right)^p\right] + (r - s)^{p - 1}\E\left[\int^r_s (L_{b,1})^p|X_u|^p + (L_{b,2})^p|Y_u|^p\,\od u\right]\\
    &\qquad + (r - s)^{p/2}\E\left[\left(\int^r_s (L_{b,3})^2|Z_u|^2\,\od u\right)^{p/2}\right]\\
    &\lesssim I_{p,b}([s,r]) + c_{\eqref{p-estimate}} [(L_{b,1})^p + (L_{b,2})^p] I_{p,T} (r - s)^p + (L_{b,3})^{p}(r - s)^{p/2}CV_p([t,T]). 
\end{align*}
For $I_2(p)$, by Burkholder--Davis--Gundy's inequalities, H\"older's inequality and Assumption \ref{coefficient:assume:general},
\begin{align*}
    I_2(p)
    &\lesssim \E\left[\left(\int^r_s |\mu^0(u)|^2 \,\od u\right)^{p/2}\right]\\
    &\qquad + (r - s)^{p/2 - 1} \E\left[\int^r_s (L_{\mu,1})^p|X_u|^p + (L_{\mu,2})^p|Y_u|^p\,\od u\right]\\
    &\qquad + \E\left[\left(\int^r_s (L_{\mu,3})^2|Z_u|^2\,\od u\right)^{p/2}\right]\\
    &\lesssim I_{p,\mu}([s,r]) + c_{\eqref{p-estimate}}\left[(L_{\mu,1})^p + (L_{\mu,2})^p\right]I_{p,T} (r - s)^{p/2} + (L_{\mu,3})^pCV_p([t,T]).
\end{align*}
Thus, we conclude that
\begin{align*}
    \E\left[|X_r - X_s|^p\right]
    &\lesssim I_1(p) + I_2(p)\\
    &\lesssim I_{p,b}([s,r]) + I_{p,\mu}([s,r])\\
    &\qquad + c_{\eqref{p-estimate}} [(L_{b,1})^p + (L_{b,2})^p] I_{p,T} (r - s)^p\\
    &\qquad + c_{\eqref{p-estimate}}\left[(L_{\mu,1})^p + (L_{\mu,2})^p\right]I_{p,T} (r - s)^{p/2}\\
    &\qquad + (L_{b,3})^{p}(r - s)^{p/2}CV_p([t,T])\\
    &\qquad + (L_{\mu,3})^pCV_p([t,T]),
\end{align*}
where the equivalence constants only relies on $p$.

Last, we turn to the estimation for $Y$ process.
By \eqref{equ:de:path-regularity} and
Theorem \ref{theo:path-regularity:BSDE}, one has
\begin{align}
    \E\left[|Y_r - Y_s|^p\right]
    &\sim \E\left[\left|Y_r - Y^{(s,r]}_r\right|^p\right] + \E\left[\left|\E\left[Y_r \big| \cG^r_s\right] - Y_s\right|^p\right]\\
    & =: II_1(p) + II_2(p).\notag
\end{align}
By the definition for the coupling variance $CV_p$, we have
\begin{align*}
    II_1(p) \leq CV_p([t,T]).
\end{align*}
For $II_2(p)$, by H\"older's inequality and Jensen's inequality, one has
\begin{align*}
    II_2(p)
    & = \E\left[\left|\E\left[Y_r - Y_s \big| \cG^r_s\right]\right|^p\right]\\
    & = \E\left[\left|\E\left[\int^r_s f(u,X_u,Y_u,Z_u)\,\od u \bigg| \cG^r_s\right]\right|^p\right]\\
    & \leq \E\left[\left|\int^r_s f(u,X_u,Y_u,Z_u)\,\od u\right|^p\right]\\
    & \lesssim \E\left[\left(\int^r_s |f^0(u)| + L_{f,1}|X_u| + L_{f,2}|Y_u| + L_{f,3}|Z_u|\,\od u\right)^p\right]\\
    & \lesssim \E\left[\left(\int^r_s |f^0(u)|\od u\right)^p\right]\\
    & \qquad + (r - s)^{p - 1} \E\left[\int^r_s (L_{f,1})^p|X_u|^p + (L_{f,2})^p|Y_u|^p\,\od u\right]\\
    & \qquad + (r - s)^{p/2} \E\left[\left(\int^r_s (L_{f,3})^2|Z_u|^2 \,\od u\right)^{p/2}\right]\\
    & \lesssim I_{p,f}([s,r]) + c_{\eqref{p-estimate}}\left[(L_{f,1})^p + (L_{f,2})^p\right]I_{p,T} (r - s)^p + (L_{f,3})^p (r - s)^{p/2} CV_p([t,T]).
\end{align*}
Combined the estimations above, we get
\begin{align*}
    \E\left[|Y_r - Y_s|^p\right]
    &\lesssim I_{p,f}([s,r])\\
    &\qquad + c_{\eqref{p-estimate}}\left[(L_{f,1})^p + (L_{f,2})^p\right]I_{p,T} (r - s)^p\\
    &\qquad + (L_{f,3})^p (r - s)^{p/2} CV_p([t,T])\\
    &\qquad + CV_p([t,T]),
\end{align*}
where the equivalence constants only relies on $p$.

\end{proof}

\smallskip

The path-regularity results for the processes $X$ and $Y$ naturally imply the 
time-regularity of the decoupling function $w$ (if $w$ exists):
\begin{rema}\label{rema:regularity:decouplingField}
    If the FBSDE \eqref{fbsde:equ} satisfies all the condition of Theorem \ref{theorem:path-regularity}
    with the case when the regular decoupling field $w$ exists, then, for all $t\leq r < s\leq T$ and $x\in\R^n$,
    the estimation for
    \begin{align*}
        \E\left[\left|w(r,x) - w(s,x)\right|^p\right],
    \end{align*}
    follows naturally from Theorem~\ref{theorem:path-regularity}. In fact, because the decoupling field $w$ is uniformly Lipschitz, thus
    \begin{align*}
        \E\left[\left|w(r,x) - w(s,x)\right|^p\right] 
        & \lesssim \E\left[\left|w(r,x) - w(s,X^{r,x}_s)\right|^p\right] + \E\left[\left|w(s,X^{r,x}_s) - w(s,x)\right|^p\right]\\
        & \lesssim \E\left[|Y^{r,x}_r - Y^{r,x}_s|^p\right] + (L_w)^p \E\left[|X^{r,x}_s - x|^p\right],
    \end{align*}
    where $(X^{r,x},Y^{r,x},Z^{r,x})$ is the solution to FBSDE \eqref{fbsde:equ} when the equation starts at time $r$ with the initial value $X_r = x\in\R^n$,
    the last inequality holds due to the properties of the regular decoupling field.
    By Theorem \ref{theorem:path-regularity}, we can get the aimed estimate.

    Moreover, if we further assume that 
    \begin{itemize}
        \item $(b,\mu,f,g)$ are deterministic functions;
        \item $(b^0,\mu^0,f^0)$ are uniformly bounded by $M\in\R^+$ on $[t,T]$.
    \end{itemize}
    Then $w$ is a deterministic function by Blumenthal's zero-one law, and 
    \begin{align*}
        \left|w(r,x) - w(s,x)\right|^p
        &\leq (c_1 + c_2|x|^p) (r - s)^{p/2},
    \end{align*}
    where the constant $c_1$, $c_2$ and $c_3$ only rely on $p$, $T - t$, all the Lipschitz constants and $M$.
    For the special case when the coefficient $\mu$ does not rely on the $Z$ process,
    our result goes back to the previous estimates obtained by A.~Fromm and P.~Imkeller in \cite{Fromm:Imkeller:13}.
\end{rema}

\subsection{$\D_{1,2}$ differentiability}
In this section, we consider the coupling function 
$$\vz_r \equiv r.$$
We write
$X^{\vz_r}$ the element in the class $\cC_{\vz_r}(X)$.

Theorem \ref{theo:Malliavin_sobolev} shows that the $2$-coupling variance can be used to characterize the
Malliavin Sobolev space $\D_{1,2}$.
For the FBSDE \eqref{fbsde:equ}, we obtain the following result.
Note that in this section, we always assume Assumption~\ref{ass:linear:diffusion}.

\begin{theo}\label{thm:Sobolev}
    If the FBSDE \eqref{fbsde:equ} satisfies all the condition of Theorem \ref{theorem:path-regularity} with $p = 2$,
    and there is a constant $M\in (0,\fz)$ such that
    \begin{equation}\label{equ:Sobolev:cond:g}
        \sup_{r\in (0,1]}\frac{\E\left[|g^{\vz_r}(X_T) - g(X_T)|^2 + U_2([t,T])\right]}{r^2} \leq M
    \end{equation}
    where $U_2((t,T])$ is defined as in Theorem \ref{thm:small:decoupling} with respect to the function $\vz_r$.
    If $\xi\in \D_{1,2}$, then, for all $s\in [t,T]$, $X_s$, $Y_s\in \D_{1,2}$.
\end{theo}
\begin{proof}
    We only show the proof for the case when the time interval is small enough.
    The other case is similar.

    For any $r\in (0,1]$, by Lemma \ref{lemm:p-estimate}, one has
    \begin{align}\label{equ:Sobolev:S}
        \frac{S_{2}([t,T])}{r^2}
        & = (T - t)\E\left[(L_{\mu,1})^2\sup_{s\in [t,T]}|X_s|^2 + (L_{\mu,2})^2\sup_{s\in [t,T]}|Y_s|^2\right]\notag\\
        &\qquad + \E\left[\int^T_t (\mu^0(s))^2\,\od s\right] < \fz,
    \end{align}
    where $S_p([t,T])$ is defined as in Theorem \ref{thm:small:decoupling}.

    Thus, by \eqref{equ:Sobolev:cond:g}, \eqref{equ:Sobolev:S}
    and Theorem \ref{thm:small:decoupling}, we conclude that
    \begin{align*}
        & \frac{\E\left[|X^{\vz_r}_s - X_s|^2\right] + \E\left[|Y^{\vz_r}_s - Y_s|^2\right]}{r^2}\\
        & \quad\lesssim \frac{\E\left[|\xi^{\vz_r} - \xi|^2\right] + \E\left[|g^{\vz_r}(X_T) - g(X_T)|^2 + U_2((t,T])\right] + S_2((t,T])}{r^2}\\
        & \quad\lesssim \sup_{r\in (0,1]}\frac{\E\left[|\xi^{\vz_r} - \xi|^2\right]}{r^2} + M + C,
    \end{align*}
    for any $r\in (0,1]$. Here the constant $C$ is the upper bound given by \eqref{equ:Sobolev:S}.
    By Theorem~\ref{theo:Malliavin_sobolev}, we complete the proof of Theorem \ref{thm:Sobolev}.
\end{proof}
\smallskip

For some special form of the coefficients, the condition \eqref{equ:Sobolev:cond:g}
can be checked without knowing the exact form of the solution:

\begin{lemm}
    If the coefficients $(b,\mu,f,g)$ are deterministic, then \eqref{equ:Sobolev:cond:g} holds.
\end{lemm}
\begin{proof}
    This follows from the fact that 
    $$U_2((t,T]) = g^{\vz_r}(X_T) - g(X_T) = 0,$$
    for all $r\in (0,1]$.
\end{proof}

\smallskip

For a more general form, we introduce the ``fractional potential condition'' introduced in \cite{Geiss:Zhou:24}.
\begin{assume}\label{assme:fractionalPotential}
    There are
    \begin{itemize}
        \item[(1)] predictable processes $V_1,V_2,V_3$ and an $\cF_T$-measurable random variable $V_4$;
        \item[(2)] deterministic and measurable functions $R_1,R_2,R_3,R_4$;
        \item[(3)] constants $\bz_i\in \R$ with $i\in\{1,2,3,4\}$, 
    \end{itemize}
    such that, for all $\tz = (x,y,z) \in \R^n\times\R^m\times\R^{m\times d}$ and coupling function $\vz$,
    the coefficient $(b,\mu,f,g)$ satisfies
    \begin{align*}
        \left|b^{\vz}(\cdot,\tz) - b(\cdot,\tz)\right| &\leq \left|V_1^{\vz}(\cdot) - V_1(\cdot)\right|^{\bz_1}R_1(\tz);\\
        \left|\mu^{\vz}(\cdot,\tz) - \mu(\cdot,\tz)\right| &\leq \left|V_2^{\vz}(\cdot) - V_2(\cdot)\right|^{\bz_2}R_2(\tz);\\
        \left|f^{\vz}(\cdot,\tz) - f(\cdot,\tz)\right| &\leq \left|V_3^{\vz}(\cdot) - V_3(\cdot)\right|^{\bz_3}R_3(\tz);\\
        \left|g^{\vz}(x) - g(x)\right| &\leq |V_4^{\vz}(\oz) - V_4(\oz)|^{\bz_4} R_4(x).
    \end{align*}
\end{assume}
Then we obtain a criterion for \eqref{equ:Sobolev:cond:g}.
\begin{lemm}
    Assume that the FBSDE \eqref{fbsde:equ} satisfies all the condition of Theorem \ref{theorem:path-regularity} with $p \geq 2$,
    and it admits a unique $\cL^p$-solution $(X,Y,Z)$.
    If its coefficients satisfy Assumption \ref{assme:fractionalPotential}, and one of the following conditions holds:
    \begin{itemize}
        \item[(\textbf{Case I})] $p > 2$,
        $R_i$ is uniformly Lipschitz in $\tz$ which does not depend on $z$, and
        \begin{align}\label{equ:condition:sobolev:lip}
            \sup_{u\in [t,T]} \sup_{r\in (0,1]} \frac{\left[\E\left[\left|V_i^{\vz_r} (u) - V_i(u)\right|^{\frac{2\bz_i p}{p - 2}}\right]\right]^{\frac{p - 2}{p}}}{r^2} < \fz,
        \end{align}
        for all $i\in \{1,2,3,4\}$.
        \smallskip

        \item[(\textbf{Case II})] $p \geq 2$, $R_i$ is uniformly bounded, and
        \begin{align}\label{equ:condition:sobolev:bounded}
            \sup_{u\in [t,T]} \sup_{r\in (0,1]} \frac{\E\left[\left|V_i^{\vz_r} (u) - V_i(u)\right|^{2\bz_i}\right]}{r^2} < \fz,
        \end{align}
        for all $i\in \{1,2,3,4\}$.
    \end{itemize}
    Then \eqref{equ:Sobolev:cond:g} holds.
\end{lemm}
\begin{proof}
    We only show the proof for \textbf{Case I}. For $r\in (0,1]$, by H\"older's inequality
    \begin{align*}
        \frac{\E\left[|g^{\vz_r}(X_T) - g(X_T)|^2\right]}{r^2}
        &\lesssim \frac{\E\left[|V_4^{\vz_r} (\oz) - V_4(\oz)|^{2\bz_4} (R_4(X_T))^2\right]}{r^2}\\
        &\lesssim \frac{\E\left[|V_4^{\vz_r} (\oz) - V_4(\oz)|^{2\bz_4} \right]}{r^2}\\
        &\qquad + \frac{\E\left[|V_4^{\vz_r} (\oz) - V_4(\oz)|^{2\bz_4} |X_T|^2\right]}{r^2}\\
        &\lesssim \frac{\left[\E\left[|V_4^{\vz_r} (\oz) - V_4(\oz)|^{\frac{2\bz_4 p}{p - 2}} \right]\right]^{\frac{p - 2}{p}}}{r^2} \left(1 + \E\left[|X_T|^p\right]\right),
    \end{align*}
    where the equivalence constants only rely on the Lipschitz constants for $R_4$.
    Moreover, by the definition for $U_2([t,T])$ and H\"older's inequality,
    it follows that
    \begin{align*}
        \E\left[U_2([t,T])\right]
        &\lesssim \int^T_t \E\left[|b^{\vz_r}(u,X_u,Y_u,Z_u) - b(u,X_u,Y_u,Z_u)|^2\right]\,\od u\\
        &\qquad + \int^T_t \E\left[|\mu^{\vz_r}(u,X_u,Y_u,Z_u) - \mu(u,X_u,Y_u,Z_u)|^2\right]\,\od u\\
        &\qquad + \int^T_t \E\left[|f^{\vz_r}(u,X_u,Y_u,Z_u) - f(u,X_u,Y_u,Z_u)|^2\right]\,\od u,
    \end{align*}
    where the equivalence constant only relies on $T$.

    By Assumption \ref{assme:fractionalPotential} and the Lipschitz property for the functions $R_1$, noticing that $R_1$ does not depend on $z$,
    we conclude that
    \begin{align*}
        &\E\left[|b^{\vz_r}(u,X_u,Y_u,Z_u) - b(u,X_u,Y_u,Z_u)|^2\right]\\
        &\quad\lesssim \left[\E\left[|V_1^{\vz_r} (\oz) - V_1(\oz)|^{\frac{2\bz_1 p}{p - 2}} \right]\right]^{\frac{p - 2}{p}}\left(1 + \E\left[|X^*_T|^p + |Y^*_T|^p\right]\right),
    \end{align*}
    where the equivalence constants only rely on the Lipschitz constants for $R_1$.
    Using the same method, we obtain the similar estimation for $\mu^{\vz_r} - \mu$ and $f^{\vz_r} - f$.
    Combining all the results, by \eqref{equ:condition:sobolev:lip}, the condition \eqref{equ:Sobolev:cond:g} holds.
\end{proof}

\section{Appendix}
\subsection{Construction of the Coupling Operator}\label{sec:Appendix:coupling}
In this part, we construct the coupling operator defined in Section~\ref{sec:coupling operator}.

    The construction of $\cC$ proceeds as follows:
    
    \begin{enumerate}
    \item[Step 1)] For any equivalence class $[f]\in L^0(R\times\Omega^0,\cR\otimes\cF^0,\rho\otimes\P^0;\R)$,
    \cite[Lemma 2.1]{Geiss:Ylinen:21} allows us to choose a representative $f$ that is $\cR\otimes\sigma(W^0_s:s\in[0,T])$-measurable.
    \item[Step 2)] Define the maps $J^i : R\times \Omega^i \to R\times\R^\bN$ by
          $$
          J^i(r,\omega^i) := \Bigl(r, \bigl(\xi^i_k(\omega^i)\bigr)_{k\in\bN}\Bigr),
          $$
          where $(\xi^i_k)_{k\in\bN}$ is an enumeration of the random variables
          $$
          \left\{\int_0^T h_\ell(s)\,\od W^i_{s,j}\,:\,\ell\in\bN,\; j=1,\ldots,N\right\},
          $$
          and $(h_\ell)_{\ell=0}^\infty$ denote the Haar functions, $L_2$ normalized to be an orthonormal basis in $L_2([0,T])$.
    \item[Step 3)] By applying \cite[Lemma 2.2]{Geiss:Ylinen:21} (the factorization lemma), we obtain a measurable function 
          \[
          \widehat{f}:R\times \R^\bN\to \R,
          \]
          such that
          \[
          f = \widehat{f}\circ J^0.
          \]
    \item[Step 4)] Define the coupling operator by setting
        \[
        \cC([f]) = \cC(f) := \Bigl[\widehat{f}\circ J^1\Bigr] \in L^0(R\times\Omega^1,\cR\otimes\cF^1,\rho\otimes\P^1;\R).
        \]
    \end{enumerate}

    \subsection{Transference of SDEs}\label{sec:Appendix:trans}
    Let $\cC_0$ and $\cC_T$ be the coupling operator defined in Section \ref{sec:coupling operator}.
    Then the following theorem holds.
    \begin{theo}[Theorem 3.3 in \cite{Geiss:Ylinen:21}]\label{theo:transference}
        Let $n,d\ge 1$ and $(\Omega,\F,\P,W,\cP_T)$ be one of the quintuples
        $$
        (\Omega^i,\F^i,\P^i,W^i,\cP_T^i),\quad i = 0,1,
        $$
        as defined in Section~\ref{sec:coupling operator}. Consider the It\^o process
        \begin{equation}\label{eqn:theo:transference}
         R_t = \xi  + \int_t^T f(s,K_s) \,\od s - \sum_{j=1}^d \int_t^T g_{j}(s,K_s) \,\od W_{s,j},
        \end{equation}
        where the following conditions are satisfied:
        \begin{enumerate}[$(S1)$]
            \item $\xi\in \cL^0(\Omega)$.
            \item The functions $f$ and each $g_j$ (for $j=1,\ldots,d$) are $(\mathcal{P_T},\mathcal{B}(C(\R^n)))$-measurable.
            \item The process $R=(R_t)_{t\in[0,T]}$, with $R_t:\Omega\to\R$, is continuous and $\F$-adapted.
            \item The process $K=(K_t)_{t\in[0,T]}$, with $K_t:\Omega\to\R^n$, is $\cP_T$-measurable.
            \item The integrability condition
            \[
            \E\int_0^T \Bigl[\,|f(s,K_s)| + |g(s,K_s)|^2\Bigr]\,\od s < \infty
            \]
            holds.
            \item The collections $(\xi,f,g,K,R,W)$ satisfies equation \eqref{eqn:theo:transference} for all $t\in[0,T]$ $\P$-a.s.
        \end{enumerate}
        
        Now, assume that the septuple $(\xi^0,f^0,g^0,K^0,R^0,W^0,\cP^0_T)$ satisfies conditions $(S1)$--$(S6)$. Let
        $\xi^1 \in \cC_0(\xi^0)$, and 
        $$
        f^1 \in \cC_T(f^0), \quad g_{j}^1 \in \cC_T(g^0_{j}),
        $$
        be $(\cP_T^1,\mathcal{B}(C(\R^n)))$-measurable modifications.
        Let
        $$
        R^1 \in \cC_T(R^0)
        $$
        be a continuous and $\F^1$-adapted process,
        and $K^1 \in \cC_T(K^0)$ be $\cP_T^1$-measurable.
        
        Then, the septuple
        $$
        (\xi^1,f^1,g^1,K^1,R^1,W^1,\cP_T^1)
        $$
        satisfies conditions $(S1)$--$(S6)$.
        \end{theo}

        \subsection{Compensate for $\cL^p$ theory for FBSDE}\label{sec:Appendix:Lp}
        \begin{proof}[Proof of Lemma \ref{lemm:p-estimate}]
            We first prove \eqref{p-estimate}.
            Consider the operator
            $$T: \cL^{\fz,p}_C([t,T]\times\Oz,\F,\P;\R^n) \to \cL^{\fz,p}_C([t,T]\times\Oz,\F,\P;\R^n)$$
            where, for any $x\in\cL^{\fz,p}_C$, $T(x) := X$ and $(X,Y,Z)$ is the solution of 
            \begin{align}\label{fbsde:equ:operator}
                \begin{cases}
                    X_s = \xi + \int^s_t b(u,x_u,Y_u,Z_u)\,\od u + \int^s_t \mu(u,x_u,Y_u,Z_u)\,\od W_u;\\
                    \ \\
                    Y_s = g(x_T) + \int^T_s f(u,x_u,Y_u,Z_u)\,\od u - \int^T_s Z_u\,\od W_u.
                  \end{cases} 
            \end{align}
            By the proof of \cite[Theorem 2.3]{Yong:20}, $T$ is well-defined, and there is $\dz > 0$ and $c_{\eqref{fbsde:operator}}\in (0,1)$ such that,
            if $T - t \leq \dz$, for any $x_1,x_2\in\cL^{\fz,p}_C$,
            \begin{equation}\label{fbsde:operator}
                \left\|T(x_1) - T(x_2)\right\|_{\cL^{\fz,p}} \leq c_{\eqref{fbsde:operator}} \left\|x_1 - x_2\right\|_{\cL^{\fz,p}}.
            \end{equation}
            For the solution of \eqref{fbsde:equ}, also denoted by $(X,Y,Z)$, then $T(X) = X$.
            Thus, \eqref{fbsde:operator} and triangle inequality show that
            \begin{equation}\label{equ:est}
                \left\|X\right\|_{\cL^{\fz,p}} \leq \frac{1}{1 - c_{\eqref{fbsde:operator}}}\left\|X^0\right\|_{\cL^{\fz,p}},
            \end{equation}
            where $(X^0,Y^0,Z^0)$ is the solution of \eqref{fbsde:equ:operator} with $x = 0$.
            By \cite[Proposition 3.2]{Briand:Delyon:03}, the processes $(Y^0,Z^0)$ satisfy
            \begin{equation}\label{bsde:est}
                \E\left[\left|(Y^0)^*_T\right|^p + \left(\int^T_t |Z^0_u|^2\,\od u\right)^{p/2}\right] \leq c_{\eqref{bsde:est}} \E\left[|g(0)|^p + \left(\int^T_t |f^0(u)|\,\od u\right)^p\right]
            \end{equation}
            where the constant $c_{\eqref{bsde:est}}$ only relies on $p$, $T$ and the Lipschitz constants $L_{f,2}, L_{f,3}$ and $L_g$.
        
            For the forward part of \eqref{fbsde:equ:operator}, by Assumption \ref{coefficient:assume:general} and BDG-inequality, it holds that
            \begin{align*}
                &\E\left[|(X^0)^*_T|^p\right]\\
                &\quad\leq 3^{p - 1}\E\left[|\xi|^p + \left(\int^T_t |b(u,0,Y^0_u,Z^0_u)|\,\od u\right)^p\right]\\
                &\quad\qquad  + 3^{p - 1}\overline{K}_p \E\left[\left(\int^T_t |\mu(u,0,Y^0_u,Z^0_u)|^2\,\od u\right)^{p/2}\right]\\
                &\quad\leq 3^{p - 1}\E\left[|\xi|^p + \left(\int^T_t |b^0(u)| + L_{b,2}|Y^0_u| + L_{b,3}|Z^0_u|\,\od u\right)^p\right]\\
                &\quad\qquad + 3^{3p/2 - 1}\overline{K}_p \E\left[\left(\int^T_t |\mu^0(u)|^2 + (L_{\mu,2})^2 |Y^0_u|^2 + (L_{\mu,3})^2 |Z^0_u|^2\,\od u\right)^{p/2}\right]\\
                &\quad\leq c_p \E\left[|\xi|^p + \left(\int^T_t |b^0(u)|\,\od u\right)^p\right]\\
                &\quad\qquad + c_p \E\left[(L_{b,2})^p T^{p}|(Y^0)^*_T|^p +  (L_{b,3})^p(T - t)^{p/2}\left(\int^T_t |Z^0_u|^2\,\od u\right)^{p/2} \right]\\
                &\quad\qquad + c_p \E\left[\left(\int^T_t |\mu^0(u)|^2\,\od s\right)^{p/2} + (L_{\mu,2})^p T^{p/2}|(Y^0)^*_T|^p + (L_{\mu,3})^p \left(\int^T_t |Z^0_u|^2\,\od u\right)^{p/2}\right]\\
                &\quad = c_p \E\left[|\xi|^p + \left(\int^T_t |b^0(u)|\,\od s\right)^p + \left(\int^T_t |\mu^0(u)|^2\,\od u\right)^{p/2}\right]\\
                &\quad\qquad + c_p \left[(L_{b,2})^pT^p + (L_{\mu,2})^pT^{p/2} + (T - t)^{p/2}(L_{b,3})^p + (L_{\mu,3})^p\right]\\
                &\quad\qquad\qquad \times \E\left[|(Y^0)^*_T|^p + \left(\int^T_t |Z^0_u|^2\,\od u\right)^{p/2}\right],
            \end{align*}
            where the equivalence constant $c_p > 0$ only relies on $p$.
            Combining this with \eqref{bsde:est} and \eqref{equ:est},
            the term $\left\|X\right\|_{\cL^{p}_C}$ can be controlled by the right-hand side of \eqref{p-estimate}.
        
            Moreover, using the similar method in the estimation for \eqref{bsde:est}, we further conclude that
            \begin{align*}
                &\E\left[\left|Y^*_T\right|^p + \left(\int^T_t |Z_u|^2\,\od u\right)^{p/2}\right]\notag\\
                &\quad \lesssim \E\left[|g(X_T)|^p + \left(\int^T_t |f(u,X_u,0,0)|\,\od s\right)^p\right]\notag\\
                &\quad \lesssim \E\left[|g(0)|^p + \left(\int^T_t |f^0(u)|\,\od u\right)^p + [(L_g)^p + (L_{f,1})^p(T - t)^p] |X^*_T|^p\right],
            \end{align*}
            where the equivalence constants only depend on $p$.
            With the estimates for $\left\|X\right\|_{\cL^{p}_C}$, we complete the proof for \eqref{p-estimate}.

            \smallskip

            Next, we prove the variance property \eqref{p-var:equ}.
            Set 
            $$\Dz X_s := X_s - \wt{X}_s,\quad \Dz Y_s := Y_s - \wt{Y}_s\quad\text{and}\quad \Dz Z_s := Z_s - \wt{Z}_s,$$
    for any $s\in [t,T]$. It holds that
    \begin{align}\label{equ:var:FBSDE}
        \begin{cases}
        \Dz X_s
        & = \xi - \wt{\xi} + \int^s_t b(u,X_u,Y_u,Z_u) - \wt{b}(u,\wt{X_u},\wt{Y_u},\wt{Z_u})\,\od u\\
        &\quad\quad + \int^s_t \mu(u,X_u,Y_u,Z_u) - \wt{\mu}(u,\wt{X_u},\wt{Y_u},\wt{Z_u})\,\od W_u;\\
        \Dz Y_s
        & = g(X_T) - \wt{g}(\wt{X}_T) + \int^T_s f(u,X_u,Y_u,Z_u) - \wt{f}(u,\wt{X_u},\wt{Y_u},\wt{Z_u})\,\od u\\
        &\quad\quad - \int^T_s \Dz Z_u\,\od W_u.
        \end{cases}
    \end{align}
    Now, for the function $h\in\{b,\mu,f\}$, we set $d_h$ the random function defined as
    \begin{equation*}
        d_h (u,x,y,z) := \wt{h}(u,x + \wt{X_u},y + \wt{Y_u},z + \wt{Z_u}) - \wt{h}(u,\wt{X_u},\wt{Y_u},\wt{Z_u}).
    \end{equation*}
    We claim that, $d_h$ admits the same properties as $\wt{h}$. In fact
    \begin{itemize}
        \item The measurability follows from the property of $\wt{h}$;
        \item $d_h(u,0,0,0) \equiv 0$;
        \item The Lipschitz property follows from the fact that
        \begin{align*}
            & \left|d_h(u,x_1,y_1,z_1) - d_h(u,x_2,y_2,z_2)\right|\\
            & \quad = \left|\wt{h}(u,x_1 + \wt{X_u},y_1 + \wt{Y_u},z_1 + \wt{Z_u}) - \wt{h}(u,x_2 + \wt{X_u},y_2 + \wt{Y_u},z_2 + \wt{Z_u})\right|\\
            & \quad \leq L_{\wt{h},1} |x_1 - x_2| + L_{\wt{h},2} |y_1 - y_2| + L_{\wt{h},3}|z_1 - z_2|.
        \end{align*}
    \end{itemize}

    We rewrite \eqref{equ:var:FBSDE} as
    \begin{align}\label{equ:var:FBSDE:re}
        \begin{cases}
        \Dz X_s 
        & = \xi - \wt{\xi} + \int^s_t d_b(u,\Dz X_u,\Dz Y_u,\Dz Z_u) + \Dz b(u,\Tz_u)\,\od u\\
        &\quad\quad + \int^s_t d_{\mu}(u,\Dz X_u,\Dz Y_u,\Dz Z_u) + \Dz \mu(u,\Tz_u)\,\od W_u;\\
        \Dz Y_s
        & = d_g(\Dz X_T) + \Dz g(X_T) + \int^T_s d_f(u,\Dz X_u,\Dz Y_u,\Dz Z_u) + \Dz f(u,\Tz_u)\,\od u\\
        &\quad\quad - \int^T_s \Dz Z_u\,\od W_u.
        \end{cases}
    \end{align}
    Thus, by \eqref{p-estimate} and \eqref{equ:var:FBSDE:re}, we obtain the estimation \eqref{p-var:equ}.
    This completes the proof of Lemma~\ref{lemm:p-estimate}.

        \end{proof}
        
        \section{Acknowledgments}
This research was supported by the Finnish centre of excellence in Randomness and Structures of the Academy of Finland, funding decision number: 346305.
The author would like to thank Prof. Stefan Geiss for his valuable advice and insightful suggestions on this research.


\begin{thebibliography}{99}

    \bibitem{Ankirchner:Formm:Wendt:22}
    S.~Ankirchner, A.~Fromm, J.~Wendt.
    \newblock A transformation method to study the solvability of fully coupled FBSDEs.
    \newblock Stochastics 94 (2022), 1-25.

    \bibitem{Antonelli:93}
    F.~Antonelli,
    \newblock Backward-forward stochastic differential equations.
    \newblock Ann. Appl. Probab. 3 (1993), 777-793.

    \bibitem{Briand:Delyon:03}
    Ph.~Briand, B.~Delyon, Y.~Hu, E.~Pardoux, L.~Stoica, 
    \newblock $L^p$ solutions of backward stochastic differential equations.
    \newblock Stochastic Process. Appl. 108 (2003), 109-129.

    \bibitem{Cvitanic:Zhang:13}
    J.~Cvitani\'c, J.~Zhang, 
    \newblock Contract Theory in Continuous-Time Models.
    \newblock Springer-Verlag Berlin Heidelberg, 2013.

    \bibitem{Karoui:Peng:Quenez:97}
    N.~El Karoui, S.~Peng, M.~C.~Quenez,
    \newblock Backward stochastic differential equations in finance.
    \newblock Math. Financ. 7 (1997), 1-71.

    \bibitem{Fromm:Imkeller:13}
    A.~Fromm, P.~Imkeller,
    \newblock Existence, uniqueness and regularity of decoupling fields to multidimensional fully coupled FBSDEs.
    \newblock arXiv:1310.0499.

    \bibitem{Geiss:Ylinen:21}
    S.~Geiss and J.~Ylinen,
    \newblock Decoupling on the Wiener space, related Besov spaces, and applications to BSDEs.
    \newblock Mem. Amer. Math. Soc. 1335, 2021.

    \bibitem{Geiss:Zhou:24}
    S.~Geiss and X.~Zhou,
    \newblock Coupling of stochastic differential equations on the Wiener space.
    \newblock arXiv:2412.10836.

    \bibitem{Hu:Peng:95}
    Y.~Hu, S.~Peng,
    \newblock Solution of forward-backward stochastic differential equations.
    \newblock Probab. Theory Relat. Fields 103 (1995), 273-283.

    \bibitem{Imkeller:Reis}
    P.~Imkeller, G.~dos Reis
    \newblock Path regularity and explicit convergence rate for BSDE with truncated quadratic growth.
    \newblock Stoch. Process. Their Appl. 120 (2010), 348-379.

    \bibitem{Imkeller:Pellat:Menoukeu-Pamen:24}
    P.~Imkeller, R.~L.~Pellat, O.~Menoukeu-Pamen,
    \newblock Differentiability of quadratic forward-backward SDEs with rough drift.
    \newblock Ann. Appl. Probab. 34 (2024), 4758-4798.

    \bibitem{Lionnet:Reis:Szpruch:15}
    A.~Lionnet, G.~dos Reis, L.~Szpruch
    \newblock Time discretization of FBSDE with polynomial growth drivers and reaction-diffusion PDEs,
    \newblock Ann. Appl. Probab. 25 (2015), 2563-2625.

    \bibitem{Ma:Protter:Yong:94}
    L.~Ma, P.~Protter, J.~Yong,
    \newblock Solving forward-backward stochastic differential equations explicitly—a four step scheme.
    \newblock Probab. Theory Relat. Fields 98 (1994), 339-359.

    \bibitem{Ma2015}
    J.~Ma, Z.~Wu, D.~Zhang, J.~Zhang,
    \newblock On well-posedness of forward-backward SDEs--A unified approach.
    \newblock Ann. Appl. Probab. 25 (2015), 2168-2214.

    \bibitem{Ma:Yong:07}
    J.~Ma, J.~Yong,
    \newblock Forward-Backward Stochastic Differential Equations and their Applications.
    \newblock Springer, 2007.

    \bibitem{Ma:Zhang:02}
    J.~Ma, J.~Zhang,
    \newblock Path regularity for solutions of backward stochastic differential equations.
    \newblock Probab. Theory Relat. Fields 122 (2002), 163-190.

    \bibitem{Nualart:06}
    D.~Nualart,
    \newblock The Malliavin Calculus and Related Topics.
    \newblock Springer, 2006.

    \bibitem{Oksendal:03}
    B.~{\O}ksendal,
    \newblock Stochastic Differential Equations: An Introduction with Applications. 6 edition
    \newblock Springer-Verlag Berlin, 2003.

    \bibitem{Oksendal:Sulem:05}
    B.~{\O}ksendal, A.~Sulem,
    \newblock Applied Stochastic Control of Jump Diffusions.
    \newblock Springer, 2005.

    \bibitem{Pardoux:Peng:90}
    E.~Pardoux and S.~Peng,
    \newblock Adapted solution of a backward stochastic differential equation.
    \newblock Syst. Control Lett. 14 (1990), 55-61.

    \bibitem{Pellat:Che-Fonka:Menoukeu-Pamen:24}
    R.L.~Pellat, E.~Che-Fonka, O.~Menoukeu-Pamen,
    \newblock Time discretization of Quadratic Forward-Backward SDEs with singular drifts.
    \newblock arXiv:2412.08497.

    \bibitem{Pellat:Menoukeu-Pamen:24}
    R.~L.~Pellat, O.~Menoukeu-Pamen,
    \newblock Density analysis for coupled forward-backward SDEs with non-Lipschitz drifts and applications.
    \newblock Stoch. Process. Their Appl. 173 (2024), 104359.

    \bibitem{Peng:93}
    S.~Peng,
    \newblock Backward stochastic differential equations and applications to optimal control.
    \newblock Appl. Math. Optim. 27 (1993), 125-144.

    \bibitem{Reisinger:Stockinger:Zhang:20}
    C.~Reisinger, W.~Stockinger, Y.~Zhang
    \newblock Path regularity of coupled McKean--Vlasov FBSDEs.
    \newblock arXiv:2011.06664.

    \bibitem{Revuz:Yor:1999}
    D.~Revuz and M.~Yor,
    \newblock Continuous martingales and Brownian motion, third edition.
    \newblock Grundlehren der Mathematischen Wissenschaften 293, Springer, 1999.

    \bibitem{Xie:Yu:20}
    J.~Yong,
    \newblock An exploration of $L_p$-theory for forward-backward stochastic differential equations with random coefficients on small durations.
    \newblock J. Math. Anal. Appl. 483 (2020), 123642.


    \bibitem{Yong:97}
    J.~Yong,
    \newblock Finding adapted solutions of forward-backward stochastic differential equations: method of continuation.
    \newblock Probab. Theory Relat. Fields 107 (1997), 537-572.

    \bibitem{Yong:20}
    J.~Yong,
    \newblock $L_p$ theory of the forward-backward stochastic differential equations.
    \newblock Banach Center Publications 122 (2020), 255-286.

    \bibitem{Yong:Zhou:99}
    J.~Yong, X.~Y.~Zhou,
    \newblock Stochastic Controls: Hamiltonian Systems and HJB Equations. Vol. 43.
    \newblock Springer Science \& Business Media, 2012.

    \bibitem{Zhang:17}
    J.~Zhang,
    \newblock Backward Stochastic Differential Equations: From Linear to Fully Nonlinear Theory.
    \newblock Springer New York, NY, 2017.

\end{thebibliography}
\end{document}